\documentclass[11pt]{article} 
\usepackage{amsmath,bbm,amssymb,a4wide,color,amsthm,enumitem,srcltx,color,setspace}
\usepackage[authoryear]{natbib}
\usepackage[colorlinks]{hyperref}
\usepackage[utf8]{inputenc}
\usepackage{aliascnt}

\usepackage{setspace}

\usepackage[draft]{fixme}
\FXRegisterAuthor{ps}{aps}{PhS}
\FXRegisterAuthor{fr}{afr}{FR}

\def\esp{\mathbb E}
\def\osc{\mathrm{osc}}
\def\pr{\mathbb P}
\def\var{\mathrm{var}}

\def\rmd{\mathrm{d}}
\def\rme{\mathrm{e}}

\newcommand\1[1]{\mathbbm{1}_{#1}}

\def\sneps{\bar{S}_n^{<\epsilon}}

\def\disc{\mathrm{Disc}}
\def\pointmass{\boldsymbol{\delta}}

\def\di{\mathcal D_I}
\def\si{\mathcal S_I}
\def\dil{\mathcal D_{I}^{\ell}}
\def\sil{\mathcal S_{I,\ell}}

\def\jl{J_1^\ell}
\def\csms{\mathcal{X}}
\def\csmssil{\mathcal{Y}_{I,\ell}}
\def\csmsdi{\mathcal{X}_I}
\def\mplusx{\mathcal M(\csms)}

\def\mplusxdil{\mathcal M(\csmsdi^\ell)}
\def\mplusxsil{\mathcal M(\csmssil)}

\def\mplusxp{\mathcal N(\csms)}
\def\mplusxf{\mathcal M_f(\csms)}

\def\mplusxpdi{\mathcal N(\csmsdi)}
\def\mplusxpdil{\mathcal N(\csmsdi^\ell)}
\def\mplusxpdisc{\mathcal N^*(\csmsdi)}
\def\mplusxpdiscl{\mathcal N^*(\csmsdi^\ell)}

\newcommand{\weakdiese}{\to_{w^{\#}}}
\newcommand\normi[1]{\left\|{#1}\right\|_{I}}
\newcommand\normil[1]{\left\|{#1}\right\|_{I,\ell}}
\newcommand{\lebesgue}{\mathrm{Leb}}

\newtheorem{theorem}{Theorem}[section]
\newaliascnt{definition}{theorem}

\aliascntresetthe{definition}

\newaliascnt{corollary}{theorem}
\newtheorem{corollary}[corollary]{Corollary}
\aliascntresetthe{corollary}

\newaliascnt{proposition}{theorem}
\newtheorem{proposition}[proposition]{Proposition}
\aliascntresetthe{proposition}

\newaliascnt{lemma}{theorem}
\newtheorem{lemma}[lemma]{Lemma}
\aliascntresetthe{lemma}

\theoremstyle{break}
\newaliascnt{example}{theorem}
\newtheorem{example}[example]{Example}
\aliascntresetthe{example}

\newaliascnt{remark}{theorem}
\newtheorem{remark}[remark]{Remark}
\aliascntresetthe{remark}

\begin{document}

\title{Convergence to stable laws in the space $D$}

\author{François Roueff\thanks{Telecom Paristech, Institut Telecom, CNRS LTCI, 46, rue
    Barrault, 75634 Paris Cedex 13, France} \and Philippe
  Soulier\thanks{Universit\'e de Paris Ouest, 92000 Nanterre,
    France. Corresponding author.}}

\date{}

\maketitle

\begin{abstract}
  We study the convergence of centered and normalized sums of i.i.d. random
  elements of the space $\mathcal{D}$ of càdlàg functions endowed with
  Skorohod's $J_1$ topology, to stable distributions in $\mathcal D$. Our
  results are based on the concept of regular variation on metric spaces and on
  point process convergence. We provide some applications, in particular to the
  empirical process of the renewal-reward process.
\end{abstract}

\section{Introduction and main results}
\label{sec:intro}
The main aim of this paper is to study the relation between regular variation in
the space $\di$ and convergence to stable processes in $\di$. Let us first describe the
framework of regular variation on metric spaces introduced by
\cite{hult:lindskog:2005,hult:lindskog:2006}. Let $I$ be a nonempty closed
subinterval of $\mathbb R$. We denote by $\di$ the set of real valued càdlàg
functions defined on $I$, endowed with the $J_1$ topology.  Let $\si$ be the
subset of $\di$ of functions $x$ such that
$$
\normi{x} = \sup_{t\in I} |x(t)| = 1 \; .
$$ 
A random element $X$ in $\di$ is said to be regularly varying if there exists
$\alpha>0$, an increasing sequence $a_n$ and a probability measure $\nu$ on
$\si$, called the spectral measure, such that
\begin{align}  
  \label{eq:reg-var-metric}
  \lim_{n\to\infty} n \, \pr \left( \normi{X} > a_n x \; , \ \frac{X}{\normi{X}}
    \in A \right) = x^{-\alpha} \nu(A) \; ,
\end{align}
for any Borel set $A$ of $\si$ such that $\nu(\partial A)=0$ where $\partial
A$ is the topological boundary of $A$. Then $\normi{X}$ has a regularly varying
right-tail, the sequence $a_n$ is regularly varying at infinity with index
$1/\alpha$ and satisfies $\pr(\normi{X}>a_n) \sim 1/n$.
\citet[Theorem~10]{hult:lindskog:2006} states that~(\ref{eq:reg-var-metric}) is
equivalent to the regular variation of the finite dimensional marginal
distributions of the process $X$ together with a certain tightness
criterion.

In the finite dimensional case, it is well-known that if $\{X_n\}$ is an
i.i.d. sequence of finite dimensional vectors whose common distribution is
multivariate regularly varying, then the sum $\sum_{i=1}^n X_i$, suitably
centered and normalized converge to an $\alpha$-stable distribution.  In
statistical applications, such sums appear to evaluate the asymptotic behavior
of an empirical estimator around its mean. Therefore we will consider centered
sums and we shall always assume that $1<\alpha<2$. The case $\alpha\in(0,1)$ is
actually much simpler. Very general results in the case $\alpha\in(0,1)$ can be
found in \cite{davydov:molchanov:zuyev:2008}.
In this case no centering is needed to
ensure the absolute convergence of the series representation of the limiting
process. In contrast if $\alpha\in(1,2)$, the centering raises additional
difficulties. This can be seen in \cite{resnick:1986}, where the point process
of exceedances has been first introduced for deriving the asymptotic behavior of
the sum $S_n=\sum_{i=1}^{n} X_{i,n}$, for $X_{i,n}=Y_i\1{[i/n,1]}$ with the
$Y_i$'s i.i.d. regularly varying in a finite-dimensional space. A thinning of
the point process has to be introduced to deal with the centering. In this
contribution, we also rely on the point process of exceedances for more general
random elements $X_{i,n}$ valued in $\di$. Our results include the case treated in
\cite[Proposition~3.4]{resnick:1986}, see \autoref{subsec:invariance}. However
they do not require the centered sum $S_n-\esp[S_n]$ to be a Martingale and the
limit process that we obtain is not a Lévy process in general, see the 
other two examples treated in \autoref{sec:applis}. Hence Martingale type arguments
as in \cite{jacod:shiryaev:2003} cannot be used.  
Our main result is the following.

\begin{theorem}
  \label{theo:conv-stable-D}
  Let $\{X_i\}$ be a sequence of i.i.d. random elements of $\di$ with the same
  distribution as $X$ and assume that~(\ref{eq:reg-var-metric}) holds with $1 <
  \alpha < 2$.  Assume moreover
\begin{enumerate}[label={\rm (A-\roman*)}]
\item \label{item:no-fix} For all $t\in I$, $\nu(\{x\in\si\;,\; t\in\disc(x)\})=0$. 
\item\label{item:tightness<epsilon} For all $\eta>0$, we have
  \begin{align}
    \label{eq:tightness<epsilon}
    \lim_{\epsilon\downarrow0} \limsup_{n\to\infty} \pr \left(  \normi{\sum_{i=1}^n
      \left(X_i \1{\{\normi{X_i} \leq a_n \epsilon\}} - \esp[X \1{\{\normi{X}
        \leq a_n \epsilon\}}]\right)}  > a_n \eta \right) = 0 \; .
  \end{align}
\end{enumerate}
Then $a_n^{-1} \sum_{i=1}^n \{X_i -\esp[X]\}$ converges weakly in $(\di,J_1)$ to
an $\alpha$-stable process $\aleph$, that admits the integral representation
\begin{align} 
 \label{eq:repres-mesure}
  \aleph(t) = c_\alpha \int_{\si} w(t) \;\rmd M(w)\;,
\end{align}
where $M$ is an $\alpha$-stable independently scattered random measure on $\si$
with control measure $\nu$ and skewness intensity $\beta\equiv1$ (totally
skewed to the right) and $c_\alpha^\alpha  = \Gamma(1-\alpha) \cos(\pi\alpha/2)$. 
\end{theorem}

\begin{remark}\label{rem:iid-no-common jumps}
For $x\in\di$, let the sets of
discontinuity points of $x$ be denoted by $\disc(x)$. If $x$ and $y$ are two
functions in $\di$, then, for the $J_1$ topology, addition may not be continuous at $(x,y)$ if
$\disc(x)\cap\disc(y)\ne\emptyset$.
  Condition \ref{item:no-fix} means that if $W$ is a random element of $\si$
  with distribution $\nu$ then, for any $t\in I$, $\pr(t\in\disc(W)) =0$; i.e.
  $W$ has no fixed jumps. See \citet[p.~286] {kallenberg:2002}.  Condition
  \ref{item:no-fix} also implies that $\nu\otimes\nu$-almost all
  $(x,y)\in\mathcal \si\times\si$, $x$ and $y$ have no common jumps.  Equivalently, if
  $W$ and $W'$ are i.i.d. random elements of $\si$ with distribution $\nu$,
  then, almost surely, $W$ and $W'$ have no common jump. This implies that if
  $W_1,\dots,W_n$ are i.i.d. with distribution $\nu$, then, almost surely,
  addition is continuous at the point $(W_1,\dots,W_n)$ in
  $(\di,J_1)^n$. Cf. \citet[Theorem~4.1]{whitt:1980}.
\end{remark}

It will be useful to extend slightly \autoref{theo:conv-stable-D} by considering
triangular arrays of independent multivariate càdlàg processes.

To deal with $\ell$-dimensional càdlàg functions, for some positive integer
$\ell$, we endow $\dil$ with $\jl$ , the product $J_1$-topology, sometimes
referred to as the weak product topology (see \cite{whitt:2002}).  We then let
$\sil$ be the subset of $\dil$ of functions $x=(x_1,\dots,x_\ell)$ such that
$$
\normil{x} = \max_{i=1,\dots,\ell}\;\;\sup_{t\in I} |x_i(t)| = 1 \; .
$$ 
Note that in the multivariate setting, we have 
$$
\disc(x)=\bigcup_{i=1,\dots,\ell}\disc(x_i)\;.
$$ 
We will prove the following slightly more general result.

\begin{theorem}
  \label{theo:conv-stable-D-array}
  Let $(m_n)$ be a nondecreasing sequence of integers tending to infinity. Let $\{X_{i,n},\, 1\leq
  i\leq m_n\}$ be an array of independent random elements of $\dil$. Assume
  that there exists $\alpha\in(1,2)$ and a probability measure $\nu$ on the
  Borel sets of $(\sil,\jl)$ such that $\nu$ satisfies
  Condition~\ref{item:no-fix} and, for all $x>0$ and Borel sets $A$ such that
  $\nu(\partial A)=0$,
  \begin{align} \label{eq:reg-var-metric-arrays-lim} \lim_{n\to\infty}
    \sum_{i=1}^{m_n} \, \pr \left( \normil{X_{i,n}} > x \; , \
      \frac{X_{i,n}}{\normil{X_{i,n}}} \in A \right) & = x^{-\alpha} \nu(A) \; , \\
    \label{eq:reg-var-metric-arrays-sup}
    \lim_{n\to\infty}\;\; \max_{i=1,\dots,m_n} \, \pr \left( \normil{X_{i,n}} > x \right) & = 0 \; , \\
    \label{eq:reg-var-metric-arrays-unif-integ}
    \lim_{x\to\infty}\limsup_{n\to\infty} \sum_{i=1}^{m_n} \, \esp \left[
      \normil{X_{i,n}}\,\1{\{\normil{X_{i,n}} > x\}}\right] & = 0 \; .
\end{align}
Suppose moreover that, for all $\eta>0$, we have
  \begin{align}
    \label{eq:tightness<epsilon-array}
    \lim_{\epsilon\downarrow0} \limsup_{n\to\infty} \pr \left(  \normil{\sum_{i=1}^{m_n}
      \left( X_{i,n} \1{\{\normil{X_{i,n}}\leq  \epsilon\}} - \esp[X_{i,n} \1{\{\normil{X_{i,n}}
        \leq \epsilon\}}]\right)}  >  \eta \right) = 0 \; .
  \end{align}
Then $\sum_{i=1}^{m_n} \{X_{i,n} -\esp[X_{i,n}]\}$ converges weakly in $(\dil,\jl)$
to an $\ell$-dimensional $\alpha$-stable process $\aleph$, that admits the integral representation
given by~(\ref{eq:repres-mesure}) with $\si$ replaced by $\sil$.
\end{theorem}
\begin{remark}
  \label{rem:iid-case}
  If $m_n=n$ and $\normil{X_{i,n}}= Y_i /a_n$ where $\{Y_i,\,i\geq1\}$ is an
  i.i.d. sequence and Condition~(\ref{eq:reg-var-metric-arrays-lim}) holds,
  then the common distribution of the random variables $Y_i$ has a regularly
  varying right tail with index $\alpha$. It follows
  that~(\ref{eq:reg-var-metric-arrays-sup}) trivially holds
  and~(\ref{eq:reg-var-metric-arrays-unif-integ}) holds by Karamata's Theorem.
  Note also that, obviously, if moreover $X_{i,n}=X_i/a_n$ with
  $\{X_i,\,i\geq1\}$ an i.i.d. sequence valued in $\dil$, then
  Condition~(\ref{eq:reg-var-metric-arrays-lim}) is equivalent to the regular
  variation of the common distribution of the $X_i$s.
\end{remark}
Hence \autoref{theo:conv-stable-D} is a special case of
\autoref{theo:conv-stable-D-array} which shall be proved in
\autoref{sec:proof-main-theo}.

We conclude this section with some comments about the $\alpha$-stable limit
appearing in \autoref{theo:conv-stable-D} (or
\autoref{theo:conv-stable-D-array}).  Its finite dimensional distributions are
defined by the integral representation~(\ref{eq:repres-mesure}) and only depend
on the probability measure $\nu$. If $X/\normi{X}$ is distributed according to
$\nu$ and is independent of $\normi{X}$, as in \autoref{subsec:stable}, then
Assumption~(\ref{eq:reg-var-metric}) holds straightforwardly and, provided that
the negligibility condition ~\ref{item:tightness<epsilon} holds, a byproduct of
\autoref{theo:conv-stable-D} is that the integral
representation~(\ref{eq:repres-mesure}) admits a version in $\di$. The existence
of c\`adl\`ag versions of $\alpha$-stable processes is also a byproduct of the
convergence in $\di$ of series representations as recently investigated by
\cite{davydov:dombry:2012} and \cite{basse:rosinski:2012}. We will come back
later to this question in~\autoref{subsec:stable} below. For now, let us state
an interesting consequence of the Itô-Nisio theorem proved in
\cite{basse:rosinski:2012}.
\begin{lemma}
  \label{lem:series-rep}
  Let $\alpha\in(1,2)$, $\nu$ be a probability measure on $\si$ and $\aleph$ be
  a process in $\di$ which admits the integral
  representation~(\ref{eq:repres-mesure}).  Let $\{\Gamma_i,\,i\geq1\}$ be the
  points of a unit rate homogeneous Poisson point process on $[0,\infty)$ and
  $\{W,\,W_i,\,i\geq1\}$ be a sequence of i.i.d. random elements of $\si$ with
  common distribution $\nu$, independent of $\{\Gamma_i\}$.  Then $\esp[W]$
  defined by $\esp[W](t)=\esp[W(t)]$ for all $t\in I$ is in $\di$ and the series
  $\sum_{i=1}^\infty \{\Gamma_i^{-1/\alpha} W_i - \esp[\Gamma_i^{-1/\alpha}] \,
  \esp[W]\}$ converges uniformly almost surely in $\di$ to a limit having the
  same finite dimensional distribution as $\aleph$.
\end{lemma}
\begin{proof}
  The fact that $\esp[W]$ is in $\di$ follows from dominated convergence and
  $\normi{w}=1$ a.s.  The finite dimensional distributions of $S_n=\sum_{i=1}^n
  \{\Gamma_i^{-1/\alpha} W_i - \esp[\Gamma_i^{-1/\alpha}] \, \esp[W_i]\}$
  converge to those of $\aleph$ as a consequence of
  \citet[Theorem~3.9]{samorodnitsky:taqqu:1994}.  Hence, to obtain the result,
  it suffices to show that the series $\sum_{i=1}^\infty \{\Gamma_i^{-1/\alpha}
  W_i - \esp[\Gamma_i^{-1/\alpha}] \, \esp[W_i]\}$ converges uniformly a.s.
  Note that the series $\sum_{i=1}^\infty \{\Gamma_i^{-1/\alpha} -
  \esp[\Gamma_i^{-1/\alpha}] \}$ converges almost surely, thus, writing
  \begin{align*}
    \sum_{i=1}^\infty \{\Gamma_i^{-1/\alpha} W_i - \esp[\Gamma_i^{-1/\alpha}] \,
    \esp[W]\} = \sum_{i=1}^\infty \Gamma_i^{-1/\alpha} \{W_i - \esp[W]\} +
    \esp[W] \sum_{i=1}^\infty \{\Gamma_i^{-1/\alpha} -  \esp[\Gamma_i^{-1/\alpha}] \} \; ,
  \end{align*}
  we can assume without loss of generality that $\esp[W]\equiv0$. Define $T_n =
  \sum_{i=1}^n i^{-1/\alpha} W_i$. By Kolmogorov's three series theorem (see
  \citet[Theorem~4.18]{kallenberg:2002}), since $\sum_{i=1}^\infty
  i^{-2/\alpha}<\infty$ and $\var(W_i(t))\leq1$, for all $t\in I$, $T_n(t)$
  converges a.s. to a limit, say, $T_\infty(t)$.

  Arguing as in \cite{davydov:dombry:2012}, we apply
  \citet[Lemma~1.5.1]{samorodnitsky:taqqu:1994} to obtain that the series
  $\sum_{i=1}^\infty |\Gamma_i^{-1/\alpha} - i^{-1/\alpha}|$ is summable.  This
  implies that the series $\Delta=\sum_{i=1}^\infty
  (\Gamma_i^{-1/\alpha}-i^{-1/\alpha})W_i$ is uniformly convergent.  Hence
  $S_n-T_n$ converges uniformly a.s. to $\Delta$ and $\Delta\in\di$.  Thus, for
  all $t\in I$, $S_n(t)$ converges a.s. to $\Delta(t)+T_\infty(t)$.  Since the
  finite distributions of $S_n$ converge weakly to those of $\aleph$, which
  belongs to $\di$ by assumption, we conclude that $\Delta+T_\infty$ has a version in $\di$.
  Hence $T_\infty$ also has a version in $\di$.

  We can now apply \citet[Theorem~2.1~(ii)]{basse:rosinski:2012} and obtain
  that, suitably centered, $T_n$ converges uniformly a.s.  Moreover, for each
  $t$, we have $\esp[T_n(t)]=\esp[T(t)]=0$ and
  $$
  \esp[|T(t)|^2]=\esp[|W(t)|^2]\sum_{i=1}^\infty i^{-2/\alpha}\leq
  \sum_{i=1}^\infty i^{-2/\alpha}\;.
  $$
  Hence $\{T(t),\,t\in I\}$ is uniformly integrable. Then
  \citet[Theorem~2.1~(iii)]{basse:rosinski:2012} shows that $T_n$ converges
  uniformly a.s. without centering. Thus $S_n$ also converges almost surely
  uniformly.
\end{proof}

 \begin{corollary}\label{cor:series-rep-aleph}
   The process $\aleph$ defined in \autoref{theo:conv-stable-D} also admits the series representation
   \begin{align}
   \label{eq:series-rep}
     \aleph(t) = \sum_{i=1}^\infty \{\Gamma_i^{-1/\alpha} W_i -
       \esp[\Gamma_i^{-1/\alpha} ] \esp[W_1]\} \; ,
   \end{align}
   where $\{\Gamma_i,W_i,\,i\geq1\}$ are as in \autoref{lem:series-rep}.
   This series is almost surely uniformly convergent.
 \end{corollary}

It seems natural to conjecture that the limit process
in~\autoref{theo:conv-stable-D} or the sum of the series in
\autoref{lem:series-rep} is regularly varying with spectral measure $\nu$ (the
distribution of the process $W$).  However, such a result is not known to hold
generally. It is proved in \citet[Section~4]{davis:mikosch:2008} under the
assumption that $W$ has almost surely continuous paths. Under an additional
tightness condition, we obtain the following result.

\begin{lemma}
  \label{lem:spectral-measure}
  Let $\alpha\in(1,2)$, $\nu$ be a probability measure on $\si$ and $W$ be a
  random element of $\si$ with distribution $\nu$. Assume that $\esp[W]$ is
  continuous on $I$ and there exist $p\in(\alpha,2]$, $\gamma>1/2$ and a
  continuous increasing function $F$ such that, for all $s<t<u$,
  \begin{subequations}
    \begin{align}
      \esp[|\bar W(s,t)|^{p}] & \leq \{F(t)-F(s)\}^\gamma \;
      , \label{eq:tightness-spectral-1} \\
      \esp[ |\bar W(s,t)\, \bar W(t,u)|^{p}] & \leq \{F(u)-F(s)\}^{2\gamma} \;
      , \label{eq:tightness-spectral-2}
    \end{align}
  \end{subequations}
  where $\bar W(s,t) = W(t)-W(s)-\esp[W(t)-W(s)]$. Then the stable process $\aleph$ defined by the integral
  representation~(\ref{eq:repres-mesure}) admits a version in $\di$ which is
  regularly varying in the sense of~(\ref{eq:reg-var-metric}), with spectral
  measure $\nu$.
\end{lemma}
\begin{remark}
  Our assumptions on $W$ are similar to those of
  \citet[Theorem~4.3]{basse:rosinski:2012} and
  \citet[Theorem~1]{davydov:dombry:2012}, with a few minor differences. For
  instance Conditions~(\ref{eq:tightness-spectral-1})
  and~(\ref{eq:tightness-spectral-2}) are expressed on a non-centered $W$ in
  these references. Here we only require $\esp[W]$ to be continuous, which,
  under~(\ref{eq:tightness-spectral-1}), is equivalent to Condition
  \ref{item:no-fix}.  Indeed take a random element $W$ in $\si$. Then, by
  dominated convergence, $\esp[W]$ is in $\di$.
  Condition~(\ref{eq:tightness-spectral-1}) implies that $W-\esp[W]$ has no pure
  jump. Thus under~(\ref{eq:tightness-spectral-1}), the process~$W$ has no pure
  jump if and only if $\esp[W]$ is continuous on $I$.
\end{remark}
\begin{proof}
  A straightforward adaptation of the arguments of the proof of
  \citet[Theorem~1]{davydov:dombry:2012} to the present context shows that the
  stable process $\aleph$ defined by~(\ref{eq:repres-mesure}) admits a version
  in $\di$. This fact is also a consequence of
  \autoref{lemma:as-convergence-serie} below, so we omit the details of the
  adaptation.

  Therefore, we only have to prove that the c\`adl\`ag version of $\aleph$,
  still denoted $\aleph$, is regularly varying in the sense
  of~(\ref{eq:reg-var-metric}), with spectral measure $\nu$. By
  \autoref{cor:series-rep-aleph}, $\aleph$ can be represented as the almost
  surely uniformly convergent series
$$
\sum_{i=1}^\infty \{\Gamma_i^{-1/\alpha} W_i - \esp[\Gamma_i^{-1/\alpha}
W_i]\}=\sum_{i=1}^\infty \Gamma_i^{-1/\alpha} \bar W_i+ \esp[W]
\sum_{i=1}^\infty \{\Gamma_i^{-1/\alpha} - \esp[\Gamma_i^{-1/\alpha}]\}\;,
$$ 
where $\bar{W}_i=W_i-\esp[W]$.
For $k\geq1$, define $\Sigma_k = \sum_{i=k}^\infty
\Gamma_i^{-1/\alpha}\bar W_i$. We proceed as in the proof of
\autoref{coro:burkholder}.  Conditioning on the Poisson process and applying
Burkholder's inequality, we have for any $t\in I$, since $\normi{W}=1$,
\begin{align}
  \label{eq:p-moment-pointwise}
\esp[|\Sigma_4(t)|^p] \leq C_p \sum_{i=4}^\infty \esp[\Gamma_i^{-p/\alpha}] <\infty \,.
\end{align}
Similarly, using the conditions~(\ref{eq:tightness-spectral-1})
and~(\ref{eq:tightness-spectral-2}), we have for some constants
$C$ and $C'$ only depending on $p$ and $\alpha$, for all $s<t<u$, 
\begin{align}
  \esp\big[\left|\Sigma_4(t)-\Sigma_4(s)\right|^p & \, \left|\Sigma_4(u)-\Sigma_4(t)\right|^p\big] \nonumber \\
  & \leq C\, \esp\left[\left(\sum_{i=4}^\infty \Gamma_i^{-2/\alpha}\right)^p\right] \, \esp[|\bar
  W(s,t)\bar W(t,u)|^p]  \nonumber \\
  \nonumber & \phantom{ = } + C\, \esp\left[\left(\sum_{i=4}^\infty
      \Gamma_i^{-p/\alpha}\right)^2\right] \esp[|\bar W(s,t)|^p] \, \esp[|\bar W(t,u)|^p] \\
 \label{eq:1} & \leq C' \{F(u)-F(s)\}^{2\gamma} \; .
\end{align}
In the first inequality we used the bounds
\begin{align*}
  \esp \Big[\Big| \sum_{i=4}^\infty \Gamma_i^{-2/\alpha} & \bar W_i(s,t) \bar W_i(t,u)\Big|^p\Big]  \\
  & \leq 2^p \sum_{i=4}^\infty \esp[\Gamma_i^{-2p/\alpha}]\esp\left[\left|\bar W(s,t)\bar
      W(t,u)-\esp[\bar W(s,t)\bar W(t,u)]\right|^p\right]  \\
  & +2^p\esp\left[\left(\sum_{i=4}^\infty \Gamma_i^{-2/\alpha}\right)^p\right] \left|\esp[\bar
    W_i(s,t) \bar W_i(t,u)]\right|^p  \\
  &\leq 2^{2p+2}\esp\left[\left(\sum_{i=4}^\infty \Gamma_i^{-2/\alpha}\right)^p\right]
  \esp\left[\left|\bar W(s,t)\bar W(t,u)\right|^p\right]\;,  \\
  \esp \Big[\Big|\sum_{i\neq j\geq4} \Gamma_i^{-1/\alpha} & \Gamma_j^{-1/\alpha} \bar W_i(s,t)\bar
  W_j(t,u)\Big|^p\Big]  \\
  & \leq \sum_{i\neq j\geq4} \esp\left[\Gamma_i^{-p/\alpha}\Gamma_j^{-p/\alpha}\right]\,\esp[|\bar
  W(s,t)|^p] \, \esp[|\bar W(t,u)|^p] \; .
\end{align*}
In the second inequality we used \autoref{lem:poissonthing}.

The bound~(\ref{eq:1}) and~(\ref{eq:p-moment-pointwise}) imply that
$\esp[\normi{\Sigma_4}^p]<\infty$, see \cite[Chapter~15]{billingsley:1968}.  Moreover, since $p/\alpha<2$ we have, for
$i=2,3$, $\esp[\Gamma_i^{-p/\alpha}]<\infty$.  Using $\normi{W_i}\leq1$ for
$i=2,3$ we finally get that $\esp[\normi{\Sigma_2}^p]<\infty$ and $Z$ can be
represented as
$$
\Gamma_1^{-1/\alpha} \bar W_1+ \Sigma_2+ \esp[W] \sum_{i=1}^\infty
\{\Gamma_i^{-1/\alpha} - \esp[\Gamma_i^{-1/\alpha}]\}
=\Gamma_1^{-1/\alpha} W_1 + T\;,
$$
where $T= \Sigma_2 +\esp[W] \sum_{i=2}^\infty \{\Gamma_i^{-1/\alpha} -
\esp[\Gamma_i^{-1/\alpha}]\}-\esp[\Gamma_1^{-1/\alpha} W_1]$ satisfies
$\esp[\normi{T}^p]<\infty$. Observe that $p>\alpha$.  Since
$\Gamma_1^{-1/\alpha}$ has a Frechet distribution with index $\alpha$, it holds
that $\Gamma_1^{-1/\alpha}W_1$ is regularly varying with spectral measure
$\nu$, which concludes the proof.
\end{proof}

In the next section, we prove some intermediate results needed to prove
Theorems~\ref{theo:conv-stable-D} and~\ref{theo:conv-stable-D-array}. In
particular, we give a condition for the convergence in $\di$ of the sequence of
expectations. This is not obvious, since the expectation fonctional is not
continuous in $\di$. We provide a criterion for the negligibility
condition~(\ref{eq:tightness<epsilon-array}) and for the sake of completeness,
we recall the main tools of random measure theory we need. In
section~\ref{sec:applis}, we give some applications of
Theorem~\ref{theo:conv-stable-D}.

\section{Some results on convergence in $\di$ and proof of the main results}

\subsection{Convergence of the expectation in $\di$}
\label{subsec:conv-expec}

It may happen that a uniformly bounded sequence $(X_n)$ converges weakly to $X$
in $(\di,J_1)$ but $\esp[X_n]$ does not converge to $\esp[X]$ in $(\di,J_1)$.
Therefore, to deal with the centering, we will need the following lemma.

\begin{lemma} 
  \label{lem:esp-cont}
  Suppose that $X_n$ converges weakly to $X$ in $(\di,J_1)$.  Suppose moreover
  that there exists $m>0$ such that $\sup_n \normi{X_n}\leq m$ a.s. and $X$ has no
  fixed jump, i.e. for all $t\in I$,
  \begin{align*} 
    \pr(t\in\disc(X)) = 0 \; .
  \end{align*}
  Then the maps $\esp[X_n]:t \to \esp[X_n(t)]$ and $\esp[X]:t \to \esp[X(t)]$
  are in $\di$, $\esp[X]$ is continuous on $I$ and $\esp[X_n]$ converges to
  $\esp[X]$ in $(\di,J_1)$.
\end{lemma}

\begin{proof}  
  Since we have assumed that $\sup_{n\geq0} \normi{X_n} \leq m$, almost surely,
  it also holds that $\normi{X}\leq m$ almost surely. The fact that $\esp[X_n]$
  and $\esp[X]$ are in $\di$ follows by bounded convergence. Because $X$ has no
  fixed jump, we also get that $\esp[X]$ is continuous on $I$.

  By Skorokhod's representation theorem, we can assume that $X_n$ converges to
  $X$ almost surely in $\di$.  By the definition of Skorokhod's metric (see
  e.g. \cite{billingsley:1968}), there exists a sequence $(\lambda_n)$ of random
  continuous strictly increasing functions mapping $I$ onto itself such that
  $\normi{\lambda_n-\mathrm{id}_I}$ and $\normi{X_n- X\circ \lambda_n}$
  converge almost surely to~0. By bounded convergence, it also holds that $ \lim_{n
    \to\infty} \esp[\normi{X_n - X \circ\lambda_n}] = 0$.  Write now
\begin{align*}
  \normi{\esp[X_n]-\esp[X]} & \leq \normi{\esp[X_n -
    X\circ\lambda_n]} + \normi{\esp[X\circ\lambda_n- X]} \; .
\end{align*}
The first term on the right-hand side converges to zero so we only consider the
second one.  Denote the oscillation of a function $x$ on a set $A$ by
\begin{align}   \label{eq:def-osc}
\osc(x;A)=\sup_{t\in A} x(t)-\inf_{t\in A} x(t) \; .
\end{align}
Let the open ball centered at $t$ with radius $r$ be denoted by $B(t,r)$.
Since $X$ is continuous at $t$ with probability one, it holds that
$\lim_{r\to0} \osc(X;B(r,t)) = 0$, almost surely. Since $\normi{X}\leq m$
almost surely, by dominated convergence, for each $t\in I$, we have
$$
\lim_{r\to0} \esp[\osc(X;B(t,r))] = 0 \;.
$$
Let $\eta>0$ arbitrary. For each $t \in I$ there exists $r(t,\eta)\in(0,\eta)>0$ such
that
$$
\esp[\osc(X;B(t,r(t,\eta)))]\leq \eta \;.
$$
Since $I$ is compact, it admits a finite covering by balls $B(t_i,\epsilon_i)$,
$i=1,\dots,p$ with $\epsilon_i =r(t_i,\eta)/2$. Fix some $\zeta\in(0,\min_{1
  \leq i \leq p} \epsilon_i)$. Then, for $s \in B(t_i,\epsilon_i)$ and by
choice of $\zeta$, we have
\begin{align*}
  | \esp[X\circ\lambda_n(s)]-\esp[X(s)] | & \leq \esp [|X\circ\lambda_n(s)-X(s)|
  \1{\{\normi{\lambda_n-\mathrm{id}_I}\leq \zeta\}}] + 2m \pr(\normi{\lambda_n-\mathrm{id}_I} > \zeta)  \\
  &  \leq \esp[ \osc(X;B(t_i,r(t_i,\eta)) ] + 2m \pr(\normi{\lambda_n-\mathrm{id}_I} > \zeta)   \\
  & \leq \eta + 2m \pr(\normi{\lambda_n-\mathrm{id}_I} > \zeta) \; .
\end{align*}
The last term does not depend on $s$, thus
$$
\normi{\esp[X\circ\lambda_n]-\esp[X]} \leq \eta + 2m
\pr(\normi{\lambda_n-\mathrm{id}_I} > \zeta) \; . 
$$ 
Since $\normi{\lambda_n-\mathrm{id}_I}$ converges almost surely to zero, we
obtain that
\begin{align*}
  \limsup_{n\to\infty} \normi{ \esp[X\circ\lambda_n]-\esp[X] } \leq \eta  \; .
\end{align*}
Since $\eta$ is arbitrary, this concludes the proof. 
\end{proof}
\begin{remark}
  Observe that, since $\esp[X]$ is continuous, the convergence
  $\esp[X_n]\to\esp[X]$ in the $J_1$ topology implies the uniform convergence so
  it is not surprising that we did obtain uniform convergence in the proof.
\end{remark}
\begin{remark}
Under the stronger assumption that $X_n$ converges uniformly to
  $X$, \autoref{lem:esp-cont} trivially holds since
  \begin{align*}
    \normi{\esp[X_n-X]} \leq \esp \left[
      \normi{X_n-X} \right] \;,
  \end{align*}
  and the result then follows from dominated convergence. If $X$ is
  a.s. continuous the uniform convergence follows from the convergence in the
  $J_1$ topology. If $X_n$ is a sum of independent variables converging weakly
  in the $J_1$ topology to $X$ with no pure jumps then the convergence in the
  $J_1$ topology again implies the uniform convergence. See
  \citet[Corollary~2.2]{basse:rosinski:2012}. However, under our assumptions,
  the uniform convergence does not always hold, see \autoref{expl:Levy} below. 
\end{remark}

\begin{remark}
  We can rephrase \autoref{lem:esp-cont}. Let $m>0$. Consider the closed
  subset of $\di$
  \begin{align*}
    \mathcal{B}_I(m)=\{x\in\mathcal{D}(I)\,,\,\normi{x}\leq m\} \; ,
  \end{align*}
  i.e. the closed ball centered at zero with radius $m$ in the uniform metric.
  Let $\mathcal{M}_0$ be the set of probability measures $\xi$ on
  $\mathcal{B}_I(m)$ such that, for all $t\in I$, $x$ is continuous at $t$
  $\xi$-a.s., i.e. 
  \begin{align*}
    \forall t \in I \; , \ \ \xi(\{x\in\mathcal{B}_I(m) \, , \, t\in\disc(x)\}) = 0 \; .
  \end{align*}
  Then \autoref{lem:esp-cont} means that the map $\xi\mapsto \int x\,\xi(\rmd
  x)$ defined on the set of probability measures on $\mathcal{B}_I(m)$ endowed
  with the topology of weak convergence takes its values in $\di$ (endowed with
  the $J_1$ topology), and is continuous on $\mathcal{M}_0$. In other words, if
  $\{\xi_n\}$ is a sequence of probability measures on $(\di,J_1)$ which
  converges weakly to $\xi$, such that $\xi_n(\mathcal{B}_I(m))=1$ and $\xi \in
  \mathcal{M}_0$, then $\int x\,\xi_n(\rmd x)$ converges to $\int x\,\xi(\rmd x)$
  in $(\di,J_1)$.

  The continuity of the map $\xi\mapsto\int x\,\xi(\rmd x)$ is not true out of
  $\mathcal{M}_0$, see \autoref{exple:J1-conter-exple} and
  \autoref{exple:M1-conter-exple} below.
\end{remark}

\begin{example}\label{expl:Levy}
  For $I=[0,1]$, set $X_n=\1{[U(n-1)/n,1]}$ and $X=\1{[U,1]}$ with $U$
  uniform on $[0,1]$. Then the assumptions of \autoref{lem:esp-cont}
  hold. However, $X_n$ converges a.s. to $X$ in the $J_1$ topology but not
  uniformly.
\end{example}

Let us now provide counter examples in the case where the assumption of
\autoref{lem:esp-cont} on the limit $X$ is not satisfied.
\begin{example}\label{exple:J1-conter-exple}
Let $I=[0,1]$, $X=\1{[1/2,1]}$
and $X_n=\1{[U_n,1]}$ where $U_n$ is drawn uniformly on $[1/2-1/n,1/2]$. Then
$X_n\to X$ a.s. in $\di$ but $\esp[X_n]$ does not converge to $\esp[X]=X$ in the
$J_1$-topology, though it does converge in the $M_1$-topology.
\end{example}
\begin{example}\label{exple:M1-conter-exple}
Set
  $X_n=\1{[u_n,1]}$ for all $n$ with probability $1/2$ and $X_n=-\1{[v_n,1]}$
  for all $n$ with probability $1/2$, where $u_n=1/2-1/n$ and
  $v_n=1/2-1/(2n)$. In the first case $X_n\to\1{[1/2,1]}$ in $\di$ with
  $I=[0,1]$ and in the second case, $X_n\to-\1{[1/2,1]}$ in $\di$. Hence
  $X_n\to X$ a.s. in $\di$ for $X$ well chosen. On the other hand, we have
  $\esp[X_n]=\1{[u_n,v_n)}$ which converges uniformly to the null function on
  $[0,u]\cup[1/2,1]$ for all $u\in(0,1/2)$, but whose sup on $I=[0,1]$ does not
  converge to 0; hence $\esp[X_n]$ cannot converge in $\di$ endowed with $J_1$,
  nor with the other usual distances on $\di$ such as the $M_1$
  distance.
\end{example}
The assumption that $\sup_n\normi{X_n}\leq m$ a.s. can be
replaced by a uniform integrability assumption. Using a truncation argument,
the following corollary is easily proved. The extension of the univariate case
to the multivariate one is obvious in the product topology so we state the
result in a multivariate setting.
\begin{corollary}
  \label{cor:esp-cont}
  Suppose that $X_n$ converges weakly to $X$ in $(\dil,\jl)$.  Suppose moreover
  that $X$ has no fixed jump and $\{\normil{X_n},\,n\geq1\}$ is uniformly
  integrable, that is,
$$
\lim_{M\to\infty}\limsup_{n\to\infty}\esp\left[\normil{X_n}
\1{\{\normil{X_n}>M\}}\right] = 0\;.
$$ 
Then the maps $\esp[X_n]:t \to \esp[X_n(t)]$ and $\esp[X]:t \to \esp[X(t)]$ are
in $\dil$, $\esp[X]$ is continuous on $I$ and $\esp[X_n]$ converges to
$\esp[X]$ in $(\dil,\jl)$.
\end{corollary}
\subsection{Weak convergence of random measures}
\label{sec:weak-conv-rand-meas}
Let $\csms$ be a complete separable metric space (CSMS).  Let $\mplusx$
denote the set of boundedly finite nonnegative Borel measures $\mu$ on
$\csms$, i.e. such that $\mu(A)<\infty$ for all bounded Borel sets $A$. A
sequence $(\mu_n)$ of elements of $\mplusx$ is said to converge weakly to
$\mu$, noted by $\mu_n\weakdiese\mu$, if $ \lim_{n\to\infty}\mu_n(f)=\mu(f) $
for all continuous functions $f$ with bounded support in $\csms$.  The
weak convergence in $\mplusx$ is metrizable in such a way that $\mplusx$ is a
CSMS, see~\citet[Theorem~A2.6.III]{daley:vere-jones:bookvolI}.  We denote by
$\mathcal B(\mplusx)$ the corresponding Borel sigma-field.

Let $(M_n)$ be a sequence of random elements of $(\mplusx,\mathcal
B(\mplusx))$.  Then, by~\citet[Theorem 11.1.VII]{daley:vere-jones:bookvolII},
$M_n$ converges weakly to $M$, noted $M_n\Rightarrow M$, if and only if
$$
(M_n(A_1),\dots,M_n(A_k)) \Rightarrow (M(A_1),\dots,M(A_k))\quad\text{in
  $\mathbb R^k$}
$$
for all $k=1,2,\dots$, and all bounded sets $A_1,\dots,A_k$ in $\mathcal
B(\mplusx)$ such that
$M(\partial A_i)=0$ a.s. for all $i=1,\dots,k$. As stated
in~\citet[Proposition~11.1.VIII]{daley:vere-jones:bookvolII}, this is equivalent
to
the pointwise convergence of the  Laplace functional of $M_n$ to that of $M$,
that is, 
\begin{equation}
  \label{eq:laplace-general}
  \lim_{n\to\infty} \esp[\rme^{-M_n(f)}]=\esp[\rme^{-M(f)}]\;,
\end{equation}
for all bounded continuous function $f$ with bounded support.

A point measure in $\mplusx$ is a measure which takes integer values on the
bounded Borel sets of $\csms$. A point process in $\mplusx$ is a random point
measure in $\mplusx$. In particular a Poisson point process has an intensity
measure in $\mplusx$.  In the following, we shall denote by $\mplusxp$ the set
of point measures in $\mplusx$ and by  $\mplusxf$ the set
of finite measures in $\mplusx$.

Consider now the space $\di$ endowed with the $J_1$ topology.  Let $\delta$ be
a bounded metric generating the $J_1$ topology on $\di$ and which makes it a CSMS,
see~\citet[Section~14]{billingsley:1968}.  From now on we denote by
$\csmsdi=(\di,\delta\wedge1)$ this CSMS, all the Borel sets of which
are bounded, since we chose $\delta$ bounded.  We further let $\mplusxpdisc$ be the subset of point measures $m$
such that, for all distinct $x$ and $y$ in $\di$ such that
$\disc(x)\cap\disc(y)\neq\emptyset$, $m(\{x,y\})<2$. In other words, $m$ is
simple (the measure of all singletons is at most 1) and the elements of the
(finite) support of $m$ have disjoint sets of discontinuity points.

\begin{lemma}\label{lem:csmsdi}
  Let $\phi:\mplusxpdi\to\di$ be defined by
$$
\phi(m) = \int w\;m(\rmd w) \;.
$$
Let $\mplusxpdi$ be endowed with the $w^\#$ topology and $\di$ with the $J_1$
topology. Then $\phi$ is continuous on $\mplusxpdisc$.
\end{lemma}
This result follows from \citet[Theorem~4.1]{whitt:1980}, which establishes the
continuity of the summation on the subset of all $(x,y)\in\di\times\di$
(endowed with the product $J_1$ topology) such that
$\disc(x)\cap\disc(y)=\emptyset$. However it requires some adaptation to the
setting of finite point measures endowed with the $w^\#$ topology. We
provide a detailed proof for the sake of completeness.
\begin{proof}
  Let $(\mu_n)$ be a sequence in $\mplusxpdi$ and $\mu\in\mplusxpdisc$ such that
  $\mu_n\weakdiese\mu$. We write $\mu=\sum_{k=1}^p\pointmass_{y_k}$, where
  $\pointmass_x$ denotes the unit point mass measure at $x$ and $p=\mu(\di)$.
  Since $\mu$ is simple, we may find $r>0$ such that, $\mu(B(y_k,r))=1$ for all
  $k=1,\dots,p$, where $B(x,r)=\{y\in\di\,:\,\delta(x,y)<r\}$. Let $g_k$ be the
  mapping defined on $\mplusxpdi$ with values in $\di$ defined by
$$
g_k(m)=
\begin{cases}
  0 &\text{ if $m(B(y_k,r/2))\neq1$,}\\
  x &\text{ otherwise}\;,
\end{cases}
$$
where, in the second case, $x$ is the unique point in the intersection of
$B(y_k,r/2)$ with the support of $m$. 

Now, for any $r'\in(0,r)$ and all  $k=1,\dots,p$, since $\partial B(y_k,r')\subset
B(y_k,r)\setminus\{y_k\}$, we have $\mu(\partial B(y_k,r))=0$ and thus
\begin{equation}
  \label{eq:inside-all-small-balls}
\lim_{n\to\infty}\mu_n(B(y_k,r'))=\mu(B(y_k,r'))=1\;.
\end{equation}
On the other hand, since $\di$ is endowed with a bounded metric, the definition
of the $w^\#$ convergence implies that
$$
\lim_{n\to\infty}\mu_n(\di)=\mu(\di)=p\;.
$$
It follows that, for $n$ large enough,
$$
\phi(\mu_n)=\sum_{k=1}^p g_k(\mu_n) \;.
$$
By \citet[Theorem~4.1]{whitt:1980}, to conclude, it only remains to show that
each term of this sum converges to its expected limit, that is, for all
$k=1,\dots,p$, $g_k(\mu_n)\to y_k$ in $\csmsdi$ as $n\to\infty$.

Using~(\ref{eq:inside-all-small-balls}), we
deduce that, for all $r'\in(0,r/2)$,
$$
\limsup_{n\to\infty}\delta(g_k(\mu_n),y_k)\leq r'\;,
$$
which shows the continuity of $g_k$ at $\mu$ and achieves the proof.  
\end{proof}

To deal with multivariate functions, we endow $\csmsdi^\ell$ with the metric 
$$
\delta_\ell((x_1,\dots,x_\ell),(x'_1,\dots,x'_\ell))= \sum_{i=1}^\ell\delta(x_i,x'_i)\;,
$$
so that the corresponding topology is the product topology denoted by $\jl$. We immediately get
the following result. 
\begin{corollary}
\label{cor:csmsdil}
  Let $\Phi:\mplusxpdil\to\dil$ be defined by
$$
\Phi(m) = \int w\;m(\rmd w) \;.
$$
Let $\mplusxpdil$ be endowed with the $w^\#$ topology and $\dil$ with the $\jl$
topology. Then $\Phi$ is continuous on $\mplusxpdiscl$.
\end{corollary}
\begin{proof}
  Let us write $\Phi=(\Phi_1,\dots,\Phi_\ell)$ with each
  $\Phi_i:\mplusxpdil\to\di$.  Since $\jl$ is the product topology, it amounts
  to show that each component $\Phi_i$ is continuous on $\mplusxpdil$. We shall
  do it for $i=1$.  Observe that the mapping $m\mapsto m_1$ defined from
  $\mplusxpdil$ to $\mplusxpdi$, by $m_1(A)=m(A\times\di^{\ell-1})$ for all
  Borel set $A$ in $\csmsdi$ is continuous for the $\weakdiese$
  topology. Moreover we have $\Phi_1(m)=\phi(m_1)$ and $m\in\mplusxpdiscl$ implies
  $m_1\in\mplusxpdisc$. Hence, by
  \autoref{lem:csmsdi}, $\Phi_1$ is continuous on $\mplusxpdiscl$, which
  concludes the proof.
\end{proof}

We now consider the space $(0,\infty]\times \sil$, which, endowed with the metric
$$
d((r,x),(r',x'))= |1/r-1/r'|+ \delta_\ell(x,x')\;,
$$
is also a CSMS. For convenience, we shall denote the corresponding metric space
by 
$$
\csmssil=\left((0,\infty]\times \sil,d\right)\;. 
$$
In this case the necessary and sufficient condition~(\ref{eq:laplace-general})
must be checked for all bounded continuous function $f$ defined on $\csmssil$
and vanishing on $(0,\eta)\times \sil$ for some $\eta>0$.

The following consequence of \autoref{cor:csmsdil} will be useful.

\begin{corollary}
\label{theo:summation}
Let $\mu \in \mplusxsil$ and $\epsilon>0$ be such that
\begin{enumerate}[label={\rm (D-\roman*)}]
\item\label{item:disc-set-summation}   for all $t\in I$, 
$\mu(\{(y,x)\in\csmssil\;,\; t\in\disc(x)\})=0$,
\item\label{item:epsilon-cond-summation}  $\mu(\{\epsilon,\infty\}\times\sil)=0$.
\end{enumerate}
Let $M$ be a Poisson point process on $\csmssil$ with control measure $\mu$. Let
$\{M_n\}$ be a sequence of point processes in $\mplusxsil$ which converges
weakly to $M$ in $\mplusxsil$. Then, the weak convergence
\begin{align}
  \label{eq:mapping-summation}
  \int_{(\epsilon,\infty)} \int_{\sil} y w \; M_n(\mathrm{d}y,\mathrm{d}w) \Rightarrow
  \int_{(\epsilon,\infty)} \int_{\sil} y w \; M(\mathrm{d}y,\mathrm{d}w) 
\end{align}
holds in $(\dil,\jl)$ and the limit has no pure jump.
\end{corollary}

\begin{proof}
Let us define the mapping $\psi:\csmssil\to\csmsdi^\ell$ by
$$
\psi(y,w) = \begin{cases}
  yw & \text{ if $y<\infty$,}\\
  0 & \text{ otherwise.}
\end{cases}
$$
Let further $\Psi:\mplusxsil\to\mplusxdil$ be the mapping defined by
$$
[\Psi(m)](A)=m\left(\psi^{-1}(A)\cap ((\epsilon,\infty)\times\sil)\right) \;,
$$
for all Borel subsets $A$ in $\mplusxdil$. Since 
$$
\partial (A\cap((\epsilon,\infty)\times\sil))\subset\partial
A\cap(\{\epsilon,\infty\}\times\sil))\;,
$$ 
we have that $m\mapsto m(\cdot\cap((\epsilon,\infty)\times\sil))$ is continuous
from $\mplusxsil$ to $\mplusxsil$  on the set
$$
\mathcal{A}_1=\left\{\mu\in\mplusxsil\,:\,\mu(\{\epsilon,\infty\}\times\sil)\right\}=0\;.
$$
Using the continuity of $\psi$ on $(0,\infty)\times\sil$, it is easy to show
that $m\mapsto m\circ\psi^{-1}$ is continuous on
$$
\mathcal{A}_2 = \left\{\mu\in\mplusxsil\,:\,\exists
  M>0,\,\mu([M,\infty]\times\sil)=0\right\}\;.
$$
Hence $\Psi$ is continuous on
$\mathcal{A}=\mathcal{A}_1\cap\mathcal{A}_2$. With~\autoref{cor:csmsdil}, we conclude that 
$$
m\mapsto  \int_{(\epsilon,\infty)} \int_{\sil} y w \; m(\mathrm{d}y,\mathrm{d}w) =\int w \; \Psi(m)(\rmd w) 
$$
is continuous as a mapping from $\Psi^{-1}(\mplusxpdi)$ endowed with the
$\weakdiese$ topology to $(\dil,\jl)$ on the set
$\Psi^{-1}(\mplusxpdisc)\cap\mathcal{A}$. Thus the weak
convergence~(\ref{eq:mapping-summation}) follows from the continuous mapping
theorem, and by observing that the sequence $(M_n)$ belongs to
$\Psi^{-1}(\mplusxpdi)$ for all $n$ and that, by
Conditions~\ref{item:disc-set-summation} and~\ref{item:epsilon-cond-summation},  
$M$ belongs to $\Psi^{-1}(\mplusxpdisc)\cap\mathcal{A}$ a.s.  (see \autoref{rem:iid-no-common jumps}).

The fact that the limit has no pure jump also follows from Condition~\ref{item:disc-set-summation}.
\end{proof}

\subsection{Convergence in $\dil$ based on point process convergence}

The truncation approach is usual in the context of regular variation to exhibit
$\alpha$-stable approximations of the empirical mean of an infinite variance
sequence of random variables. 

The proof relies on separating small jumps and big jumps and rely on point
process convergence.  In the following result, we have gathered the main steps
of this approach. To our knowledge, such a result is not available in this
degree of generality.

\begin{theorem}
  \label{theo:francois}
  Let $\{N_n,\,n\geq1\}$ be a sequence of finite point processes on $\mathcal X$
  and $N$ be a Poisson point process on $\csmssil$ with mean measure
  $\mu$. Define, for all $n\geq1$ and $\epsilon>0$,
  \begin{align*}
    S_n &= \int_{(0,\infty)} \int_{\sil} y w \; N_n(\mathrm{d}y,\mathrm{d}w) \;
    , \ \ S_n^{<\epsilon} = \int_{(0,\epsilon]} \int_{\sil} y w \;
    N_n(\mathrm{d}y,\mathrm{d}w) \; , \\
    Z_\epsilon& = \int_{(\epsilon,\infty)} \int_{\sil} y w \;
    N(\mathrm{d}y,\mathrm{d}w) \; ,
  \end{align*}
  which are well defined in $\dil$ since $N$ and $N_n$ have finite supports in
  $(\epsilon,\infty)\times\sil$ and $(0,\infty)\times\sil$, respectively.
  Assume that the following assertions hold.
\begin{enumerate}[label={\rm (B-\roman*)}]
\item\label{item:thm-ppp-hyp-conv} $N_n\Rightarrow N$ in $(\mplusxsil,\mathcal
  B(\mplusxsil))$.
\item\label{item:thm-ppp-hyp-disc-set} For all $t\in I$,
  $\mu(\{(y,x)\in\csmssil\;,\; t\in\disc(x))\})=0$ and
  $\mu(\{\infty\}\times\sil)=0$.
\item\label{item:moment-levy} $\displaystyle\int_{(0,1]}
  y^2\;\mu(\rmd{y},\sil)<\infty$.
\item\label{item:thm-ppp-hyp-unif-int} For each $\epsilon>0$, the sequence
  $\left\{\int_{(\epsilon,\infty)} y \;N_n(\mathrm{d}y,\sil), n\geq1\right\}$
  is uniformly integrable.
\item\label{item:thm-ppp-hyp-negl} The following negligibility condition holds~:
  for all $\eta>0$,
  \begin{equation}
    \label{eq:ppp-negligable}
    \lim_{\epsilon\downarrow0} \limsup_{n\to\infty} \pr \left(
      \normil{S_n^{<\epsilon}-\esp[S_n^{<\epsilon}]} > \eta \right) = 0\;,
  \end{equation}
\end{enumerate}
Then  the following assertions hold.
\begin{enumerate}[label={\rm (C\arabic*)}]
\item\label{item:thm-ppp-conc1} For each $\epsilon>0$, $Z_\epsilon \in \dil$,
  $\esp[Z_\epsilon]\in\dil$ and $Z_\epsilon-\esp[Z_\epsilon]$ converges weakly
  in $(\dil,\jl)$ to a process $\bar Z$ as $\epsilon\to0$.
\item\label{item:thm-ppp-conc2} $S_n-\esp[S_n]$ converges weakly in $(\dil,\jl)$
  to $\bar Z$.
\end{enumerate}
\end{theorem}

\begin{proof}
  For $\epsilon>0$, we define
$$  
S_n^{>\epsilon} = \int_{(\epsilon,\infty)} \int_{\sil} y w
N_n(\mathrm{d}y,\mathrm{d}w), \quad \bar S_n^{>\epsilon} = S_{n}^{>\epsilon} -
\esp[S_n^{>\epsilon}] \; , \; \bar S_n^{<\epsilon} = S_{n}^{<\epsilon} -
\esp[S_n^{<\epsilon}] \; ,
$$
which are random elements of $\dil$.  By \autoref{theo:summation} and
Conditions~\ref{item:thm-ppp-hyp-conv} and~\ref{item:thm-ppp-hyp-disc-set}, we
have that $S_n^{>\epsilon}$ converges weakly in $(\dil,\jl)$ to $Z_\epsilon$, provided that
  $\mu(\{\epsilon\}\times\sil)=0$,  and
$Z_\epsilon$ has no pure jump.

Since $\normil{S_n^{>\epsilon}}\leq \int_{(\epsilon,\infty)} y \, N_n(\rmd
y,\sil)$, by Condition~\ref{item:thm-ppp-hyp-unif-int}, we get that
$\{\normil{S_n^{>\epsilon}},\,n\geq1\}$ is uniformly integrable.  Applying
\autoref{cor:esp-cont}, we get that $\esp[ S_n^{>\epsilon}]$ converges to
$\esp[Z_\epsilon]$ in $(\dil,\jl)$ and that $\esp[Z_\epsilon]$ is continuous on
$I$.  Thus addition is continuous at $(Z_{\epsilon},\esp[Z_{\epsilon}])$. See
\citet[p.~84]{whitt:2002}.  We obtain that, for all $\epsilon>0$, as
$n\to\infty$,
\begin{equation}
  \label{eq:truncation-easy-part}
  \bar S_n^{>\epsilon}\Rightarrow Z_{\epsilon}-\esp[Z_\epsilon]  \quad \text{ in $\di$}\;.   
\end{equation}
Define $\bar S_n = S_n-\esp[S_n]$. Then
\begin{equation}
  \label{eq:S_n-trunc-decomp}
 \bar S_n = \bar S_n^{>\epsilon} + \bar S_n^{<\epsilon} \;,
\end{equation}
and~(\ref{eq:ppp-negligable}) can be rewritten as
\begin{equation}
  \label{eq:Bill-trick}
  \lim_{\epsilon\to0}\limsup_{n\to\infty}\pr(\normil{\bar S_n - \bar S_n^{>\epsilon}}>\eta) =
  0\;.
\end{equation}
By \citet[Theorem~4.2]{billingsley:1968}, Assertion~\ref{item:thm-ppp-conc1}
and~(\ref{eq:Bill-trick}) imply Assertion~\ref{item:thm-ppp-conc2}.  Hence, to
conclude the proof, it remains to prove Assertion~\ref{item:thm-ppp-conc1}, that
is, $\bar Z_{\epsilon}$ converges weakly in $(\dil,\jl)$ to a process $\bar Z$.  For
all $t\in I$ and $0<\epsilon<\epsilon'$, we have
$$
Z_{\epsilon}(t)-Z_{\epsilon'}(t)=\int_{(\epsilon,\epsilon']}\int_{\sil}y\,w(t)\;N(\rmd y,\rmd w)\;,
$$
where $N$ is a Poisson process with intensity measure $\mu$. Thus, denoting by
$|a|$ the Euclidean norm of vector $a$ and by $\mathrm{Tr}(A)$ the trace of
matrix $A$, we have
\begin{align*}
  \esp[|\bar Z_{\epsilon}(t)-\bar Z_{\epsilon'}(t)|^2] & = \mathrm{Tr} \left(
    \mathrm{Cov}(Z_{\epsilon}(t) - Z_{\epsilon'}(t)) \right)  \\
  & = \int_{(\epsilon,\epsilon']} \int_{\sil} y^2 \, |w(t)|^2 \; \mu(\rmd y,\rmd w) \\
  &\leq \ell \, \int_{(\epsilon,\epsilon']} y^2 \; \mu(\rmd y,\sil) \; .
\end{align*}
We deduce from~\ref{item:moment-levy} that $\bar Z_{\epsilon}(t)-\bar Z_{1}(t)$
converges in $L^2$ as $\epsilon$ tends to $0$. Thus there exists a process $\bar
Z$ such that $\bar Z_{\epsilon}$ converges to $\bar Z$ pointwise in probability,
hence in the sense of finite dimensional distributions.  To obtain the
convergence in $(\dil,\jl)$, since we use the product topology in $\dil$, it
only remains to show the tightness of each component. Thus, hereafter we assume
that $\ell=1$.  Denote, for $x\in\di$ and $\delta>0$,
\begin{align} 
   \label{eq:def-w"}
  w^{\prime\prime}(x,\delta) &= \sup\{|x(t)-x(s)| \wedge |x(u)-x(t)| \; ; \ s
  \leq t \leq u\in I \; , \ |u-s| \leq \delta\}\;.
\end{align}
By \citet[Theorem~15.3]{billingsley:1968}, it is sufficient to
prove that, for all $\eta>0$,
\begin{align}
  \label{eq:tight-condzero}
  \lim_{A\to\infty} \sup_{0<\epsilon\leq1} \pr(\normi{\bar{Z}_{\epsilon}} > A)  = 0 \; , \\
  \label{eq:tightness-module}
  \lim_{\delta\downarrow0} \limsup_{\epsilon\downarrow0}  \pr(w^{\prime\prime}(\bar{Z}_{\epsilon},\delta) >\eta) = 0 \; ,\\
  \label{eq:tightness-module-bordg}
  \lim_{\delta\downarrow0} \limsup_{\epsilon\downarrow0}  \pr(\osc(\bar{Z}_{\epsilon};[a,a+\delta)) >\eta) = 0 \; ,\\
  \label{eq:tightness-module-bordd}
  \lim_{\delta\downarrow0} \limsup_{\epsilon\downarrow0}  \pr(\osc(\bar{Z}_{\epsilon};[b-\delta,b)) >\eta) = 0 \; ,
\end{align}
where $I=[a,b]$ and $\osc$ is defined in~(\ref{eq:def-osc}).  We start by proving~(\ref{eq:tight-condzero}).  For any
$\epsilon_0\in(0,1]$ and $\epsilon\in[\epsilon_0,1]$, we have
$\normi{\bar{Z}_\epsilon} \leq \int_{(\epsilon_0,\infty)} y N(\rmd y,\si)$,
whence
\begin{align*}
  \sup_{\epsilon_0\leq\epsilon\leq1} \pr(\normi{\bar{Z}_{\epsilon}} > A)
  \leq A^{-1} \esp\left[ \int_{(\epsilon_0,\infty)} y N(\rmd y,\si) \right]  \; ,
\end{align*}
which is finite by Condition~\ref{item:thm-ppp-hyp-unif-int}.

This yields that $\lim_{A\to\infty} \sup_{\epsilon_0\leq\epsilon\leq1}
\pr(\normi{W_{\epsilon}} > A) = 0$ and to conclude the proof
of~(\ref{eq:tight-condzero}), we only need to show that, for any $\eta>0$,
\begin{equation}
  \label{eq:epseps0}
  \lim_{\epsilon_0\downarrow0} \sup_{0<\epsilon<\epsilon_0}
  \pr\left(\normi{\bar{Z}_{\epsilon}-\bar{Z}_{\epsilon_0}}>\eta\right) =0\;.
\end{equation}
The arguments leading to~(\ref{eq:truncation-easy-part}) can be used to show
that, for all $0<\epsilon<\epsilon_0$,
\begin{equation}
  \label{eq:conv-of-increments}
  \bar S_n^{>\epsilon_0} - \bar S_n^{>\epsilon}\Rightarrow \bar
  Z_{\epsilon_0}-\bar Z_{\epsilon}\quad\text{in $\di$}\;.  
\end{equation}
(although the latter is not a consequence of~(\ref{eq:truncation-easy-part})
because $Z_{\epsilon_0}$ and $Z_{\epsilon}$ have common jumps).  By definition
(see~(\ref{eq:S_n-trunc-decomp})), we have $ \bar{S}_n^{<\epsilon_0} -
\bar{S}_n^{<\epsilon} = \bar{S}_n^{>\epsilon} - \bar{S}_n^{>\epsilon_0}$.
By~(\ref{eq:conv-of-increments}) and the continuous mapping theorem, we get that
$\normi{\bar S_n^{<\epsilon_0} - \bar S_n^{<\epsilon}} \Rightarrow
\normi{\bar{Z}_{\epsilon_0} - \bar{Z}_{\epsilon}}$. Thus, by the
Portmanteau Theorem, for all $\eta>0$,
\begin{align*}
  \pr(\normi{\bar{Z}_{\epsilon_0}-\bar{Z}_{\epsilon}} \geq \eta) & =
  \limsup_{n\to\infty} \pr(\normi{\bar{S}_n^{<\epsilon_0} - \bar{S}_n^{<\epsilon}}\geq  \eta) \\
  & \leq \limsup_{n\to\infty} \pr(\normi{\bar{S}_n^{<\epsilon_0}} \geq
  \eta/2) + \limsup_{n\to\infty} \pr(\normi{\bar{S}_n^{<\epsilon}} \geq
  \eta/2) \; .
\end{align*}
We conclude by applying Condition~\ref{item:thm-ppp-hyp-negl} which precisely
states that both terms in the right-hand side tend to zero as $\epsilon_0$ tends
to~0, for any $\eta>0$.  This yields~(\ref{eq:epseps0})
and~(\ref{eq:tight-condzero}) follows.

Define now the modulus of continuity of a function $x\in\di$ by
\begin{align*}
  w(x,\delta) = \sup\{|x(t)-x(s)| \; ,  \ s,  t \in I  \; ,  \ |t-s|\leq \delta \} \; .
\end{align*}
We shall rely on the fact that, for any $x,y\in\di$, 
\begin{align*} 
  w^{\prime\prime}(x+y,\delta) \leq w^{\prime\prime}(x,\delta) + w(y,\delta) \; .
\end{align*}
Note that this inequality is no longer true if $w(y,\delta)$ is replaced by
$w^{\prime\prime}(y,\delta)$.  We get that, for any $0<\epsilon<\epsilon_0$ and
$\delta>0$,
\begin{align}
  \nonumber w^{\prime\prime}(\bar{Z}_{\epsilon},\delta) &\leq
  w^{\prime\prime}(\bar{Z}_{\epsilon_0},\delta) +  w(\bar{Z}_{\epsilon}-\bar{Z}_{\epsilon_0},\delta) \\
  \label{eq:modZeps} & \leq w^{\prime\prime}(\bar{Z}_{\epsilon_0},\delta) +
  2\normi{\bar{Z}_{\epsilon}-\bar{Z}_{\epsilon_0}} \; .
\end{align}
Since $\bar{Z}_{\epsilon_0}$ is in $\di$, we have, for any fixed $\epsilon_0>0$,
$$
\lim_{\delta\to0}\pr\left(w^{\prime\prime}(\bar{Z}_{\epsilon_0},\delta)>\eta\right) = 0 \;.
$$
Hence, with~(\ref{eq:epseps0}), we conclude that~(\ref{eq:tightness-module})
holds. Similarly, since, for each subinterval $T$, we have
$$
\osc(\bar{Z}_{\epsilon};T) \leq \osc(\bar{Z}_{\epsilon_0};T) + 2 \,
\normi{\bar{Z}_{\epsilon} - \bar{Z}_{\epsilon_0}} \; ,
$$
so we obtain~(\ref{eq:tightness-module-bordg})
and~(\ref{eq:tightness-module-bordd}). This concludes the proof.
\end{proof}

\subsection{Regular variation in $\mathcal D$ and point process convergence}
\label{sec:reg-var-pp}

Let now $\{X_{i,n},\,1\leq i\leq m_n\}$ be an array of independent random
elements in $\di$ and define the point process of exceedances $N_n$ on
$(0,\infty]\times\sil$ by
\begin{align}
  \label{eq:Nndef}
  N_n = \sum_{i=1}^{m_n} \pointmass_{\normil{X_{i,n}},\frac{X_{i,n}}{\normi{X_{i,n}}}}
    \; ,
\end{align}
with the convention that $\pointmass_{0,0/0}$ is the null mass.  If the
processes $X_{n,i}$, $1 \leq i \leq m_n$ are i.i.d. for each $n$, then it is
shown in \citet[Theorem~2.4]{dehaan:lin:2001} that
Condition~(\ref{eq:reg-var-metric}) implies the convergence of the sequence of
point processes $N_n$ to a Poisson point process on $\di$. We slightly extend
here this result to triangular arrays of vector valued processes.

Let $N$ be a Poisson point process on $\csmssil=(0,\infty]\times\sil$ (see
\autoref{sec:weak-conv-rand-meas}) with mean measure $\mu_\alpha$ defined by
$\mu_\alpha(\rmd y\rmd w)= \alpha y^{-\alpha-1}\mathrm d y\nu(\rmd w)$.

\begin{proposition}\label{prop:RVpoitmeasureconv-array}
  Conditions~(\ref{eq:reg-var-metric-arrays-lim})
  and~(\ref{eq:reg-var-metric-arrays-sup}) in
  \autoref{theo:conv-stable-D-array} imply the weak
  convergence of $N_n$ to $N$ in  $(\mplusxsil,\mathcal B(\mplusxsil))$.
\end{proposition}
\begin{proof}
 As explained in \autoref{sec:weak-conv-rand-meas}, we only need to check
  the convergence
  \begin{equation}
    \label{eq:laplace-conv}
    \lim_{n\to\infty}\esp [\rme^{-N_n(f)}] = \esp \rme^{-N(f)} 
  \end{equation}
  for all bounded continuous functions $f$ on $(0,\infty]\times\sil$
  and vanishing on $(0,\epsilon)\times\sil$ for some $\epsilon>0$. Consider such
  a function $f$. We have
$$
\log \esp [\rme^{-N_n(f)}] =  \sum_{i=1}^{m_n}  \log \left(1 +
  \esp\left[g(\normil{X_{i,n}},X_{i,n}/\normil{X_{i,n}})\right]\right)\;,
$$
where we denoted $g:(y,w)\mapsto\rme^{-f(y,w)}-1$, which is continuous and
bounded with the same support as $f$. Moreover $g$ is lower bounded by some
constant $A>-1$. It follows that there exists a positive constant $C$ such that
\begin{multline*}
   \left|\log \left(1 +
    \esp\left[g(\normil{X_{i,n}},X_{i,n}/\normil{X_{i,n}})\right]\right) -
  \esp\left[g(\normil{X_{i,n}},X_{i,n}/\normil{X_{i,n}})\right]\right|\\
\leq
C\, \pr^2(\normil{X_{i,n}}>\epsilon) \; .
\end{multline*}
Define  $\delta_n =\max_{i=1,\dots,n}\pr(\normil{X_{i,n}}>\epsilon)$. Then 
\begin{align*}
  \left| \log \esp [\rme^{-N_n(f)}] -  \sum_{i=1}^{m_n}  
  \esp\left[g(\normil{X_{i,n}},X_{i,n}/\normil{X_{i,n}})\right]\right| 
\leq C \delta_n \sum_{i=1}^{m_n}  \pr(\normil{X_{i,n}}>\epsilon) \; .
\end{align*}
Since $\delta_n=o(1)$ by~(\ref{eq:reg-var-metric-arrays-sup}) and
$\sum_{i=1}^{m_n}  \pr(\normil{X_{i,n}}>\epsilon)=O(1)$ by~(\ref{eq:reg-var-metric-arrays-lim}), we obtain 
\begin{align*}
  \log \esp [\rme^{-N_n(f)}] =   \sum_{i=1}^{m_n}  
  \esp\left[g(\normil{X_{i,n}},X_{i,n}/\normil{X_{i,n}})\right] + o(1)\;.
\end{align*}
We conclude by applying (\ref{eq:reg-var-metric-arrays-lim}) again. 
\end{proof}

\subsection{A criterion for negligibility}

Condition~\ref{item:thm-ppp-hyp-negl} is a negligibility condition in the
sup-norm. It can be checked separately on each component of
$S_n^{<\epsilon}-\esp[S_n^{<\epsilon}]$.  We give here a sufficient condition
based on a tightness criterion.

\begin{lemma}
  \label{lem:negl-via-w''} 
  Let $\{U_{\epsilon,n},\,\epsilon>0,n\geq1\}$ be a collection of random elements in $\di$ such that, for all $t\in I$,
\begin{equation}
  \label{eq:var-negl-cond}
  \lim_{\epsilon\to0}\limsup_{n\to\infty} \var(U_{\epsilon,n}(t)) = 0 \;,
\end{equation}
and, for all $\eta>0$,
\begin{equation}
  \label{eq:tightness<epsilonbis-generic}
  \lim_{\delta\to0} \sup_{0<\epsilon\leq1} \limsup_{n\to\infty} 
  \pr(w^{\prime\prime}(U_{\epsilon,n},\delta) >\eta) =0 \;,  
\end{equation}
where $w^{\prime\prime}$ is defined in~(\ref{eq:def-w"}). Then  $\{U_{\epsilon,n},\,\epsilon>0,\,n\geq1\}$
satisfies the negligibility condition~\ref{item:thm-ppp-hyp-negl}, that is, for all $\eta>0$,
  \begin{equation}
    \label{eq:ppp-negligable-generic}
    \lim_{\epsilon\to0} \limsup_{n\to\infty} \pr \left(
      \normi{U_{\epsilon,n}} > \eta \right) = 0\;.
  \end{equation} 
\end{lemma}

\begin{proof}
By~(\ref{eq:var-negl-cond}) and the Bienaimé-Chebyshev inequality, we get that,
for all $\eta>0$ and $t\in I$,
$$
\lim_{\epsilon\to0}\limsup_{n\to\infty} \pr(|U_{\epsilon,n}(t)|>\eta) =0 \; .
$$
It follows that, for any $p\geq1$, $t_1<\dots<t_p$ and $\eta>0$,
\begin{align}
  \label{eq:conv-p-points}
  \lim_{\epsilon\to0} \limsup_{n\to\infty} \pr(\max_{k=1,\dots,p} 
|U_{\epsilon,n}(t_k)| > \eta) = 0 \; .
\end{align}
Fix some $\zeta>0$. By Condition~(\ref{eq:tightness<epsilonbis-generic}), we
can choose $\delta>0$ such that $\limsup_{n\to\infty}
\pr(w^{\prime\prime}(U_{\epsilon,n},\delta) >\eta) \leq \zeta$ for all
$\epsilon\in(0,1]$.  Note now that, as in \cite[Proof of Theorem~15.7,
Page~131]{billingsley:1968}, for any $\delta>0$, we may find an integer
$m\geq1$ and $t_1<t_1<\dots<t_m$, such that for all $x\in \mathcal D$
\begin{align}
  \label{eq:borne-sup-x}
  \normi{x} \leq w^{\prime\prime}(x,\delta) + \max_{k=1,\dots,m}|x(t_k)| \; .
\end{align}
Thus, by~(\ref{eq:conv-p-points}), we obtain
\begin{multline*}
  \lim_{\epsilon\to0} \limsup_{n\to\infty}  \pr \left(
    \normi{U_{\epsilon,n}} > \eta \right) \\
   \leq \sup_{0<\epsilon\leq1} \limsup_{n\to\infty}
   \pr(w^{\prime\prime}(U_{\epsilon,n},\delta) >\eta) 
   + \lim_{\epsilon\to0} \limsup_{n\to\infty}
  \pr(\max_{k=1,\dots,p} |U_{\epsilon,n}(t_k)| > \eta/2) \leq \zeta \; ,
\end{multline*}
which concludes the proof since $\zeta$ is arbitrary.
\end{proof}

In order to obtain~(\ref{eq:tightness<epsilonbis-generic}), we can use 
the tightness criteria of \citet[Chapter~15]{billingsley:1968}. We then get
the following corollary.

\begin{corollary}
   \label{coro:burkholder}
   Let $X,X_i,i\geq1$, be i.i.d. random elements in $\di$ such that $\normi{X}$
   is regularly varying with index $\alpha\in(1,2)$. Let $\{a_n\}$ be an
   increasing sequence such that
   $\lim_{n\to\infty}n\pr(\normi{X}>a_n)=1$. Assume that there exist
   $p\in(\alpha,2]$, $\gamma>1/2$ and a continuous
   increasing function $F$ on $I$ and a sequence of increasing functions $F_n$
   that converges pointwise (hence uniformly) to $F$ such that
  \begin{subequations}
    \begin{align}
      \sup_{0<\epsilon\leq1} n a_n^{-p }
      \esp[|\bar{X}_{\epsilon,n}(s,t)|^{p}] & \leq \{F_n(t)-F_n(s)\}^{\gamma}
      \; , \label{eq:burk1}      \\
      \sup_{0<\epsilon\leq1} n^{2} a_n^{-2p}
      \esp[|\bar{X}_{\epsilon,n}(s,t)|^{p}|\bar{X}_{\epsilon,n}(t,u)|^p] &
      \leq \{F_n(u)-F_n(s)\}^{2\gamma} \; , \label{eq:burk2}
    \end{align}
  \end{subequations}
where
$$
\bar{X}_{\epsilon,n}(s,t) = \{X(t)-X(s)\}\1{\normi{X}\leq a_n\epsilon} -
\esp\left[ \{X(t)-X(s)\}\1{\normi{X}\leq a_n\epsilon} \right]\;.
$$  
Then the negligibility condition~(\ref{eq:ppp-negligable-generic}) holds with 
$$
U_{\epsilon,n}=a_n^{-1}\sum_{i=1}^n \left\{X_i\1{\normi{X_i}\leq
    a_n\epsilon}-\esp\left[X_i\1{\normi{X_i}\leq a_n\epsilon}\right]\right\}\;.
$$
\end{corollary}

\begin{proof}
  We apply \autoref{lem:negl-via-w''}. By the regular variation of $\normi{X}$,
  it holds that
  \begin{align*}
    \limsup_{n\to\infty} \var(U_{\epsilon,n}(t)) \leq \limsup_{n\to\infty}
    n a_n^{-2} \esp[\normi{X}^2 \1{\{\normi{X} \leq a_n \epsilon\}}] =
    O(\epsilon^{2-\alpha}) \; ,
  \end{align*}
which yields~(\ref{eq:var-negl-cond}).

By Burkholder's inequality, \cite[Theorem~2.10]{hall:heyde:1980}, and
conditions~(\ref{eq:burk1}) and~(\ref{eq:burk2}), we have, for some constant $C>0$, for any
$\epsilon\in(0,1]$, 
  \begin{align*} 
    \esp[|U_{\epsilon,n}(s,t)|^{p} |U_{\epsilon,n}(t,u)|^{p}] & \leq C
    \,
    n a_n^{-2p} \, \esp[|\bar{X}_{\epsilon,n}(s,t)|^p|\bar{X}_{\epsilon,n}(t,u)|^p] \\
    & \phantom{ \leq } + C \, n^2 a_n^{-2p} \, \esp[|\bar{X}_{\epsilon,n}(s,t)|^p] \,
    \esp[|\bar{X}_{\epsilon,n}(t,u)|^p] \\
    &  \leq  2\,C \, \{F_n(u) - F_n(s)\}^{2\gamma} \; .
  \end{align*}
By Markov's inequality, this yields 
  \begin{align*} 
    \pr\left(|U_{\epsilon,n}(s,t)|>\lambda \; ,
      |U_{\epsilon,n}(t,u)| > \lambda \right) 
    & \leq C\,\lambda^{-2p}\, \{F_n(u) - F_n(s)\}^{2\gamma} \; .
  \end{align*}
  Arguing as in the proof of \citet[Theorem~15.6]{billingsley:1968} (see also
  the proof of the Theorem of~\citet[p.~335]{genest:ghoudi:remillard:1995}), we
  obtain, for $\delta>0$ and $\eta>0$
  \begin{align*}
    \pr(w^{\prime\prime}(\sneps,\delta)>\eta) \leq C'\, \eta^{-2p}
    \sum_{i=1}^K \{F_n(t_i)-F_n(t_{i-1})\}^{2\gamma} \; ,
  \end{align*}
  where $C'$ is a constant which depends neither on $\epsilon\in(0,1]$ nor on
  $\delta>0$, $K>1/(2\delta)$ and $t_1<\cdots<t_K$ are such that
  $|t_i-t_{i-1}|\leq \delta$ and $I \subset \cup_{i=1}^n [t_{i-1},t_i]$.  Thus
  \begin{align*}
    \sup_{0<\epsilon\leq1} \limsup_{n\to\infty}
    \pr(w^{\prime\prime}(\sneps,\delta)>\eta) \leq C''\, \eta^{-2p}
    \{w(F,\delta)\}^{2\gamma-1} \; ,
  \end{align*}
  where $C''$ does not depend on $\delta$.  Since $F$ is continuous and
  $\gamma>1/2$, this yields~(\ref{eq:tightness<epsilonbis-generic}).
\end{proof}

\subsection{Proof of \autoref{theo:conv-stable-D-array}}
\label{sec:proof-main-theo}
We apply \autoref{theo:francois} to the point processes $N_n$ and $N$
defined in \autoref{sec:reg-var-pp} and the measure $\mu_\alpha$ in lieu of
$\mu$.  By \autoref{prop:RVpoitmeasureconv-array}, we have that $N_n$ converges
weakly to $N$ in $\mplusxsil$, i.e. Condition~\ref{item:thm-ppp-hyp-conv} holds.
Condition~\ref{item:no-fix} and the definition of $\mu_\alpha$ imply~\ref{item:thm-ppp-hyp-disc-set}.
Condition~(\ref{eq:tightness<epsilon-array}) corresponds to Condition~\ref{item:thm-ppp-hyp-negl}.
Condition~\ref{item:moment-levy} is a consequence of the definition of
$\mu_\alpha$:
\begin{align*}
  \int_{(0,1]} y^2\;\mu_\alpha(\rmd{y},\sil) = \int_0^1 \alpha y^{2-\alpha-1} \rmd t
  = \frac\alpha{2-\alpha} \; .
\end{align*}

For $0<\epsilon<x$, define $Y_{n} = \int_{(\epsilon,\infty)} y\;N_n(\rmd
y,\sil)$ and $Y = \int_{(\epsilon,\infty)} y\;N(\rmd y,\sil)$. 
The weak convergence of $N_n$ to $N$ implies that of $N_n(\cdot\times\sil)$ to
$N(\cdot\times\sil)$. In turn, by continuity of the map $m\mapsto
\int_\epsilon^\infty y \;m(\rmd y)$ on the set of point measures on
$(0,\infty]$ without mass on $\{\epsilon,\infty\}$, the weak convergence of
$N_n(\cdot\times\sil)$ to $N(\cdot\times\sil)$ implies that of $Y_{n}$ to $Y$.

Let $\sigma_n$ be the measure on $(0,\infty]$ defined by $ \sigma_n((x,\infty])
= \sum_{i=1}^{m_n} \pr(\normil{X_{i,n}} >x)$. We have
\begin{align*}
\esp[Y_{n}]=\int_{(\epsilon,x]} y \, \sigma_n(\rmd y) + \sum_{i=1}^{m_n}
  \esp\left[\normi{X_{i,n}}\1{\normil{X_{i,n}}>x}\right]\;.
\end{align*}
Condition~(\ref{eq:reg-var-metric-arrays-lim}) implies that $\sigma_n$ converges
vaguely on $(0,\infty]$ to the measure with density $\alpha x^{-\alpha-1}$ with
respect to Lebesgue's measure. Thus 
$$
\lim_{n\to\infty}\int_{(\epsilon,x]} y \, \sigma_n(\rmd y)= \int_{(\epsilon,x]}
y \, \alpha y^{-\alpha-1}\rmd y \;.
$$
The last two displays and Condition~(\ref{eq:reg-var-metric-arrays-unif-integ})
imply that $\esp[Y_n]$ converges to $\esp[Y]$. and $\esp[Y]<\infty$. Since
$Y_n,Y$ are non-negative random variables, this implies the uniform
integrability of $\{Y_n\}$. We have proved~\ref{item:thm-ppp-hyp-unif-int} for
$\epsilon$ such that $\mu(\{\epsilon\}\times\sil)=0$. By monotony with respect
to $\epsilon$, this actually holds for any $\epsilon>0$.

Finally, the representation~(\ref{eq:repres-mesure}) follows
from~\citet[Theorem~3.12.2]{samorodnitsky:taqqu:1994}.

\section{Applications}
\label{sec:applis}

The usual way to prove the weak convergence of a sum of independent regularly
varying functions in $\di$ is to establish the convergence of finite-dimensional
distributions (which follows from the finite-dimensional regular variation) and
a tightness criterion.  We consider here another approach, based on functional
regular variation.  It has been proved in \autoref{sec:reg-var-pp} that
functional regular variation implies the convergence of the point process of
(functional) exceedances.  Thus, in order to apply \autoref{theo:conv-stable-D}
or \autoref{theo:conv-stable-D-array}, an asymptotic negligibility condition
(such as (\ref{eq:tightness<epsilon}) or~(\ref{eq:tightness<epsilon-array}),
respectively) must be proved.  Since the functional regular variation condition
takes care of the ``big jumps'', the negligibility condition concerns only the
``small jumps'', i.e.~we must only prove the tightness of sum of truncated
terms. This can be conveniently done by computing moments of any order
$p>\alpha$, even though they are infinite for the original series.  We provide
in this section some examples where this new approach can be fully carried out.

\subsection{Invariance principle}
\label{subsec:invariance}
We start by proving that the classical invariance principle is a particular case
of \autoref{theo:francois}. Let $\{z_i\}$ be a sequence of i.i.d.~random
variables in the domain of attraction of an $\alpha$-stable law, with
$\alpha\in(1,2)$. Let $a_n$ be the $1/n$-th quantile of the distribution of
$|z_1|$ and define the partial sum process $S_n$ by
\begin{align*}
  S_n(t) = a_n^{-1} \sum_{k=1}^{[nt]} (z_k-\esp[z_1]) \; .
\end{align*}
For $u\in[0,1]$, denote by $w_u$ the indicator function of the interval $[u,1]$
i.e. $w_u(t)=\1{[u,1]}(t)$ and define $X_{k,n} = a_n^{-1} z_kw_{k/n}$. Then we
can write $S_n = \sum_{k=1}^{n} (X_{k,n}-\esp[X_{k,n}])$. We will apply
\autoref{theo:conv-stable-D-array} to prove the convergence of $S_n$ to a stable
process in $\mathcal{D}(I)$ with $I=[0,1]$.  Note that
$\normi{X_{k,n}}=z_k/a_n$. Thus, by \autoref{rem:iid-case}, we only
need to prove that~(\ref{eq:reg-var-metric-arrays-lim}) holds with a measure
$\nu$ that satisfies Condition \ref{item:no-fix} and the negligibility
condition~(\ref{eq:tightness<epsilon-array}). Let~$\nu$ be the probability
measure defined on $\si$ by $\nu(\cdot ) 
= \int_0^1 \pointmass_{w_u}(\cdot) \, \rmd u$,
$\mu_\alpha$ be defined on $(0,\infty]\times\si$ by $\mu_\alpha((r,\infty]
\times \cdot) = r^{-\alpha} \nu(\cdot)$ and $\mu_n$ be the measure in the
left-hand side of~(\ref{eq:reg-var-metric-arrays-lim}). Since
$\normi{X_{k,n}}=z_k/a_n$, and the random variables $z_k$ are i.i.d. and
$w_{k/n}$ are deterministic, we have, for all $r>0$ and Borel subsets $A$ of
$\si$,
\begin{align*}
  \mu_n((r,\infty]\times A) = \left(n \pr(z_1>a_n r)\right) \times \left(\frac1n
    \sum_{k=1}^n\1{\{w_{\frac k n} \in A\}} \right) \; .
\end{align*}
By the regular variation of $z_1$, the first term of this product converges to
$r^{-\alpha}$.  The second term of this product can be written as $P_n \circ
\phi^{-1}(A)$, where $P_n = n^{-1} \sum_{k=1}^n \pointmass_{k/n}$ is seen as a
probability measure on the Borel sets of $[0,1]$ and $\phi:[0,1]\to \di$ is
defined by $\phi(u)=w_u$. Since $\phi$ is continuous (with $\di$ endowed by
$J_1$) and $P_n$ converges weakly to the Lebesgue measure on $[0,1]$, denoted by
$\lebesgue$, by the continuous mapping theorem, we have that $P_n\circ
\phi^{-1}$ converges weakly to $\lebesgue\circ\phi^{-1}=\nu$.  This proves
that~(\ref{eq:reg-var-metric-arrays-lim}) holds.

To prove Condition~(\ref{eq:tightness<epsilon-array}), note that
\begin{align*}
  \normi{S_n^{<\epsilon}} = a_n^{-1} \max_{1\leq k \leq n}\left| \sum_{i=1}^k
    \left(z_k \1{\{|z_k|\leq a_n \epsilon\}}-\esp\left[z_k \1{\{|z_k|\leq a_n
          \epsilon\}}\right]\right)\right| \; ,
\end{align*}
where $S_n^{<\epsilon}$ denotes the sum appearing in the left-hand
side of~(\ref{eq:tightness<epsilon-array}).
By Doob's inequality \cite[Theorem~2.2]{hall:heyde:1980},
we obtain
\begin{align*}
  \esp[ \|S_n^{<\epsilon}\|_\infty^2] &\leq 2 \;\var\left(a_n^{-1} \sum_{i=1}^n
    z_k \1{\{|z_k|\leq a_n \epsilon\}}\right) \leq
  na_n^{-2}\esp[z_1^2\1{\{|z_k|\leq a_n \epsilon\}}] = O(\epsilon^{2-\alpha})
  \; ,
\end{align*}
by regular variation of $z_1$.  This bound and Markov's inequality
yield~(\ref{eq:tightness<epsilon-array}).

\subsection{Stable processes}
\label{subsec:stable}
Applying \autoref{coro:burkholder}, we obtain a criterion for the convergence of
partial sums of a sequence of i.i.d. processes that admit the representation
$RW$, where $R$ is a Pareto random variable and $W\in\si$. This type of process
is sometimes referred to as (generalized) Pareto processes. See
\cite{ferreira:dehaan:zhou:2012}.  

\begin{proposition} 
  \label{lemma:as-convergence-serie}
  Let $\{R,R_i\}$ be a sequence of i.i.d. real valued random variables in the
  domain of attraction of an $\alpha$-stable law, with $1 < \alpha <2$. Let
  $\{W,W_i,i\geq1\}$ be an i.i.d. sequence in $\si$ with distribution $\nu$
  satisfying the assumptions of \autoref{lem:spectral-measure}, and independent
  of the sequence $\{R_i\}$.  Then, defining $a_n$ as an increasing sequence
  such that by $\pr(R>a_n)\sim 1/n$, $a_n^{-1} \sum_{i=1}^n \{R_iW_i-\esp[R]
  \esp[W]\}$ converges weakly in $\di$ to a stable process $Z$ which admits the
  representation~(\ref{eq:repres-mesure}).
\end{proposition}

\begin{remark}
  By \autoref{lem:series-rep}, the stable process $Z$ also admits the series
  representation~(\ref{eq:series-rep}), which is almost surely convergent
  in~$\di$ and by \autoref{lem:spectral-measure} it is regularly varying in the
  sense of~(\ref{eq:reg-var-metric}), with spectral measure $\nu$. As mentioned
  in the proof of \autoref{lem:spectral-measure}, the proof we give here of the
  existence of a version of $Z$ in $\di$ is different from the proof of
  \cite{davydov:dombry:2012} or \cite{basse:rosinski:2012}.
\end{remark}

\begin{proof}[Proof of \autoref{lemma:as-convergence-serie}]
  We apply \autoref{theo:conv-stable-D} to $X_i = R_iW_i$. The regular variation
  condition~(\ref{eq:reg-var-metric}) holds trivially since $\normi{X}=R$
  is independent of $X/\normi{X}=W$.
  Condition~(\ref{eq:tightness-spectral-2}) implies that $W$ has no fixed jump,
  i.e. Condition~\ref{item:no-fix} holds.  Thus we only need to prove that the
  negligibility condition~\ref{item:tightness<epsilon} holds.  Write $S_n^{<\epsilon} =
  a_n^{-1} \sum_{i=1}^n \{R_{i}\1{\{R_i\leq \epsilon a_n\}}W_i -
  \esp[R\1{\{R\leq a_n \epsilon\}}] \esp[W]\}$ and $r_{n,i} =
  a_n^{-1}R_i\1{\{R_i\leq a_n\epsilon\}}$. Then,
  \begin{align}
    \label{eq:conv-serie-centrage}
    S_n^{<\epsilon} = \sum_{i=1}^n r_{n,i} \{W_i- \esp[W]\} + \esp[W]
    \sum_{i=1}^n \{r_{n,i}-\esp[r_{n,i}] \} \; .
  \end{align}
  Since $\normi{\esp[W]}\leq1$, the second term's infinite norm on $I$ can be
  bounded using the Bienaymé-Chebyshev inequality and the regular variation of
  $R$ which implies, for any $p>\alpha$, as $n\to\infty$,
    \begin{equation}
      \label{eq:rv-bound-conv-serie}
      \esp[|r_{n,i}|^p]\sim \frac\alpha{p-\alpha}\;\epsilon^{p-\alpha}\;n^{-1}\;.
    \end{equation}
    Hence we only need to deal with the first term in the right-hand side
    of~(\ref{eq:conv-serie-centrage}), which is hereafter denoted by $\tilde
    S_n^{<\epsilon}$. Since $R$ is independent of $W$,
    Conditions~(\ref{eq:burk1}) and~(\ref{eq:burk2}) are straightforward
    consequences of~(\ref{eq:tightness-spectral-1}) and~(\ref{eq:tightness-spectral-2}). Thus
    Condition~\ref{item:tightness<epsilon} holds by
    \autoref{coro:burkholder}. The last statement follows from
    \autoref{lem:spectral-measure}
\end{proof}

\subsection{Renewal reward process}
\label{subsec:renewal}

Consider a renewal process $N$ with i.i.d. interarrivals $\{Y_i,i \geq1\}$ with
common distribution function $F$, in the domain of attraction of a
stable law with index $\alpha\in(1,2)$. Let $a_n$ be a norming sequence defined
by $a_n = F^\leftarrow(1-1/n)$. Then, for all $x>0$,
\begin{align*}
  \lim_{n\to\infty} n \bar F(a_nx) = x^{-\alpha} \; .
\end{align*}
Consider a sequence of rewards $\{W_i,i\geq1\}$ with distribution function $G$ and
define the renewal reward process $R$ by
\begin{align*}
  R(t) = W_{N(t)} \; .
\end{align*}
let $\phi$ be a measurable function and define $A_T(\phi)$ by
\begin{align*}
  A_T(\phi) = \int_0^T \phi(R(s)) \, \mathrm d s \; .
\end{align*}
We are concerned with the functional weak convergence of $A_T$. We moreover
assume that the sequence $\{(Y,W),\,(Y_i,W_i),i \geq1\}$ is i.i.d. and that
 $Y$ and $W$ are asymptotically independent in the sense of
\cite{maulik:resnick:rootzen:2002}, i.e.
\begin{align} \label{eq:asympt-indep}
  \lim_{n\to\infty} n \pr\left(\left(\frac Y{a_n},W\right) \in \cdot\right)
  \stackrel {v} \to \mu_\alpha \otimes G^*
\end{align}
on $]0,\infty] \times \mathbb R$, where $G^*$ is a probability measure on
$\mathbb R$. This assumptions is obviously satisfied when $Y$ and $W$ are
independent, with $G^*=G$ in that case. 
 
When $Y$ and $W$ are independent and $\esp[|\phi(W)|^\alpha]<\infty$, it has
been proved by \cite{taqqu:levy:1986} that $a_T^{-1} \{A_T(\phi) -
\esp[A_T(\phi)]\}$ converges weakly to a stable law. 

Define $\lambda = (\esp[Y])^{-1}$ and 
\begin{align*}
  F_0(w) = \lambda \esp[Y\1{\{W \leq w\}}] \; .
\end{align*}
Then $F_0$ is the steady state marginal distribution of the renewal reward process and
$\lim_{t\to\infty} \pr(R(t)\leq w) = F_0(w)$.
For $w\in\mathbb{R}$, consider the functions $\1{\{\cdot\leq w\}}$, which
yields the usual one-dimensional empirical process:
\begin{align*}
  E_T(w) = a_T^{-1} \int_0^T \{\1{\{R(s)\leq w\}} - F_0(w)\} \, \mathrm{d} s  \;  .
\end{align*}

\begin{theorem}
  Assume that~(\ref{eq:asympt-indep}) holds with $G^*$ continuous. The sequence
  of processes $E_T$ converges weakly in $\mathcal{D}(\mathbb{R})$ endowed with
  the $J_1$ topology as $T$ tends to infinity to the process~$E^*$ defined~by
  \begin{align*}
    E^*(w) = \int_{-\infty}^\infty \{\1{\{x\leq w\}} - F_0(w)\} M(\mathrm{d} x) \; ,
  \end{align*}
  where $M$ is a totally skewed to the right stable random measure with
  control measure $G^*$, i.e.
  \begin{align*}
     \log \esp \left[ \mathrm e^{\mathrm i t \int_{-\infty}^\infty \phi(w)
         M(\mathrm d w)} \right] = - |t|^\alpha \lambda c_\alpha
     \esp[|\phi(W^*)|^\alpha] \{ 1 + \mathrm i \, \mathrm{sign}(t)\beta(\phi)
     \tan(\pi\alpha/2)\} \; ,
  \end{align*}
  where $W^*$ is a random variable with distribution $G^*$,
  $c_\alpha^\alpha=\Gamma(1-\alpha)\cos(\pi\alpha/2)$ and $\beta(\phi) =
  \esp[\phi_+^\alpha(W^*)]/\esp[|\phi(W^*)|^\alpha]$.
  
\end{theorem}

\begin{remark}
  Equivalently, $E^*$ can be expressed as $E^* = Z \circ G^* - F_0 \cdot Z(1)$, where
  $Z$ is a totally skewed to the right Lévy $\alpha$-stable process with
  characteristic function $\esp \left[ \mathrm e^{\mathrm i t Z(1)} \right] =
  \exp\{- |t|^\alpha \lambda c_\alpha \{ 1 + \mathrm i \, \mathrm{sign}(t)
  \tan(\pi\alpha/2)\}$.  If moreover $Y$ and $W$ are independent, then the
  marginal distribution of $R(0)$ is $G$, $G^*=G$ and the limiting
  distribution can be expressed as $Z \circ G - G Z(1)$ and thus the law of
  $\sup_{w\in\mathbb R} E^*(w)$ is independent of $G$.
\end{remark}

\begin{proof}
Write
\begin{subequations}
  \begin{align}
    E_T(w) & = a_T^{-1} \sum_{i=0}^{N(T)} Y_i \1{\{W_i\leq w\}} + a_T^{-1}
    \{T-S_{N(T)}\}\1{\{W_N(T)\leq w\}} - a_T^{-1} \lambda T
    \esp[Y\1{\{W\leq w\}}]  \nonumber \\
    & = a_T^{-1} \sum_{i=0}^{N(T)} \{Y_i  \1{\{W_i\leq w\}}  -
    \esp[Y\1{\{W\leq w\}}\} 
    -  a_T^{-1} \{ S_{N(T)}-\lambda^{-1} N(T)\} F_0(w)  \label{eq:main-empiric-2} \\
    & \phantom{ = } - a_T^{-1} \{S_{N(T)}-T\} \{ \1{\{W_{N(T)}\leq w\}})
    - \lambda\esp[Y\1{\{W\leq w\}}]\} \; . \label{eq:remainder}
  \end{align}
\end{subequations}
The term in~(\ref{eq:remainder}) is $o_P(1)$, uniformly with respect to $w \in
\mathbb{R}$.  Define $U_i=G^*(W_i)$ and $U=G^{*}(W)$. Define the sequence of
bivariate processes $S_n$ on $I=[0,1]$ by
\begin{align*}
  S_n(t) = a_n^{-1} \sum_{i=1}^{n} \left( Y_i [\1{\{U_i\leq t\}},1]' - \esp[Y[\1{\{U_i
      \leq t\}},1]'] \right) \; , 
\end{align*}
where $x'$ denotes the transpose of a vector $x\in\mathbb{R}^2$.  Then the term
in~(\ref{eq:main-empiric-2}) can be expressed as the scalar product
$[1,-F_0(w)]S_{N(T)}(G^*(w))$. Using that $N(T)/T$ converges almost surely to
$\lambda$, we can relate the asymptotic behavior of $S_{N(T)}$ to that of $S_n$.
The latter is obtained by applying \autoref{theo:conv-stable-D-array}. We prove
that the assumptions hold in two steps.
\begin{enumerate}[label={(\alph*)}]
\item Let $\phi$ be the mapping $(y,w)\mapsto y
  \left[\1{[G^*(w),1]},\,\1{[0,1]}\right]'$. This mapping is continuous from
  $(0,\infty)\times\mathbb{R}$ to $\di^2$. Thus,~(\ref{eq:asympt-indep})
  implies that the distribution of $Y
[\1{\{U\leq t\}},1]'$ is regularly varying with index $\alpha$ in $\dil$ with
$\ell=2$ and $\nu$ defined by
\begin{align*}
  \nu(\cdot) = \pr((\1{[U^*,1]},\1{[0,1]})' \in \cdot)
\end{align*}
where $U^*$ is uniformly distributed on
$[0,1]$. Conditions~(\ref{eq:reg-var-metric-arrays-lim}),~(\ref{eq:reg-var-metric-arrays-sup})
and~(\ref{eq:reg-var-metric-arrays-unif-integ}) then follow by
Remark~\ref{rem:iid-case}.

\item We must next prove the asymptotic negligibility
  condition~(\ref{eq:tightness<epsilon-array}).  It suffices to prove it for
  the first marginal $X=Y\1{[U,1]}$. For $\epsilon>0$ and $n\geq1$, define
  $G_{n,\epsilon} (t) = na_n^{-2} \epsilon^\alpha \esp[Y^2\,\1{\{Y \leq a_n
    \epsilon\}}\,\1{\{U \leq t\}}]$.  It follows that~(\ref{eq:burk1})
  and~(\ref{eq:burk2}) hold with $p=2$, $\gamma=1$ and
  $F_{n}=\sup_{0<\epsilon\leq1}G_{n,\epsilon}=G_{n,1}$. Moreover, by
  Assumption~(\ref{eq:asympt-indep}) and Karamata's Theorem, we have
  $\lim_{n\to\infty} G_{n,1} (t) = t$. Thus we can apply
  \autoref{coro:burkholder} to obtain~(\ref{eq:tightness<epsilon-array}).
\end{enumerate}
By \autoref{theo:conv-stable-D-array}, the previous steps imply that $S_n$
converges weakly in $(\mathcal D,J_1)$ to a bivariate stable process which can
be expressed as $[Z,Z(1)]$, where $Z$ is a totally skewed to the right
$\alpha$-stable L\'evy process.
\end{proof}

For completeness, we  state the following result. 
\begin{proposition} 
  \label{prop:fidi} 
  Under Assumption~(\ref{eq:asympt-indep}), $a_T^{-1} \{A_T(\phi) -
  \esp[A_T(\phi)]\}$ converges weakly to a stable law which can be expressed as
  $\int_{-\infty}^\infty \{\phi(w) -\lambda\esp[Y\phi(W)]\} M(\mathrm d w)$,
  where $M$ is a totally skewed to the right stable random measure with
  control measure $G^*$, i.e.
  \begin{align}
     \label{eq:log-laplace}
     \log \esp \left[ \mathrm e^{\mathrm i t \int_{-\infty}^\infty \phi(w)
         M(\mathrm d w)} \right] = - |t|^\alpha \lambda c_\alpha
     \esp[|\phi(W^*)|^\alpha] \{ 1 + \mathrm i \, \mathrm{sign}(t)\beta(\phi)
     \tan(\phi\alpha/2)\} \; ,
  \end{align}
  where $W^*$ is a random variable with distribution $G^*$, $\beta(\phi) =
  \esp[\phi_+^\alpha(W^*)]/\esp[|\phi(W^*)|^\alpha]$.
\end{proposition}

\appendix
\section{A useful lemma}
\begin{lemma}
  \label{lem:poissonthing}
  Let $\{\Gamma_i,\,i\geq1\}$ be the  points of a unit rate homogeneous Poisson
  point process on $[0,\infty)$. Then for any $1<\alpha<p\leq2$, we have
$$
\esp\left[\left(\sum_{i=4}^\infty
\Gamma_i^{-p/\alpha}\right)^2\right] <\infty \;.
$$
\end{lemma}
\begin{proof}
  Observe that
$$\sum_{i=4}^\infty
\Gamma_i^{-p/\alpha} \leq 
\sum_{i=1}^\infty
\Gamma_i^{-p/\alpha} \1{\{\Gamma_i\geq1\}}+\sum_{i=4}^\infty
\Gamma_i^{-p/\alpha} \1{\{\Gamma_i<1\}}\;.
$$
The first term has finite second moment since $\int_1^\infty x^{-kp/\alpha}\rmd
x<\infty$ for $k=1,2$. The second term satisfies
$$
\sum_{i=4}^\infty
\Gamma_i^{-p/\alpha} \1{\{\Gamma_i<1\}}\leq \Gamma_4^{-p/\alpha}
\sum_{i=4}^\infty \1{\{\Gamma_i<1\}} \leq
\Gamma_4^{-p/\alpha}\left(1+\sum_{i=5}^\infty
  \1{\{\Gamma_i-\Gamma_4<1\}}\right) \;.
$$ 
Since $2p/\alpha<4$ we have $\esp[\Gamma_4^{-2p/\alpha}]<\infty$. Since  $\sum_{i=5}^\infty
  \1{\{\Gamma_i-\Gamma_4<1\}}$ is a Poisson variable independent of $\Gamma_4$,
  the proof is concluded.
\end{proof}


\end{document}